\documentclass[smallcondensed]{svjour3}      
\smartqed  
\usepackage{graphicx}

\usepackage{amssymb,amsmath, color}
\usepackage{amssymb}
\usepackage{amsbsy}
\usepackage{enumitem,algorithm2e,algorithmic}

\usepackage{ulem}

\usepackage[colorlinks=true]{hyperref}

\newcommand{\R}{{\mathbb R}}

\DeclareMathOperator{\argmin}{argmin}

\newcommand{\cC}{{\mathcal C}}
\newcommand{\cE}{{\mathcal E}}
\newcommand{\cH}{{\mathcal H}}
\newcommand{\cO}{{\mathcal O}}

\newcommand{\demi}{\frac{1}{2}}
\newcommand{\ie}{{\it i.e.}\,\,}

\newlength{\textlarg} 

\newcommand{\eqdef}{:=}

\newcommand{\norm}[1]{\left\|{#1}\right\|}

\usepackage{ulem}

\begin{document}

\title{Fast convex optimization via inertial dynamics combining viscous and Hessian-driven damping with time rescaling}

\titlerunning{Inertial dynamics with Hessian damping and time rescaling}

\author{Hedy ATTOUCH \and A\"icha BALHAG   \and Zaki CHBANI   \and Hassan RIAHI}

\institute{
Hedy ATTOUCH  \at IMAG, Univ. Montpellier, CNRS, Montpellier, France\\
hedy.attouch@umontpellier.fr,\\  Supported by COST Action: CA16228
\and
A\"icha BALHAG   \and Zaki CHBANI   \and Hassan RIAHI\\
 Cadi Ayyad University \\ S\'emlalia Faculty of Sciences 
 40000 Marrakech, Morroco\\
 aichabalhag@gmail.com \and   chbaniz@uca.ac.ma  \and h-riahi@uca.ac.ma 
}
\maketitle


\begin{abstract}
  In a  Hilbert setting, we develop fast methods for convex unconstrained optimization.  We rely on the asymptotic behavior of an inertial system combining geometric damping with temporal scaling.
The convex function to  minimize  enters the dynamic via its gradient. The dynamic includes three coefficients varying with time, one is a viscous damping coefficient, the second is attached to the Hessian-driven damping, the third is a time scaling coefficient. We study the  convergence rate of the values  under general conditions involving the  damping  and  the time scale coefficients.
 The obtained results are based on a new Lyapunov analysis and they encompass known results on the subject. We pay particular attention to the case of an asymptotically vanishing viscous damping, which is directly related to the accelerated gradient method of Nesterov. The Hessian-driven damping significantly reduces the oscillatory aspects. As a main result, we obtain an exponential rate of convergence of values without assuming the strong convexity of the objective function.
The temporal discretization of these dynamics opens the gate to a large class of inertial optimization algorithms.
\end{abstract}





\section{Introduction}
Throughout the paper, $\mathcal{H}$ is a real Hilbert space with inner product $\langle \cdot,\cdot\rangle$ and  induced norm $\Vert \cdot \Vert,$
and $f: \mathcal{H} \rightarrow  \R$ is a  convex and differentiable  function. We aim at developping fast numerical methods for solving
the  optimization problem 
\begin{equation*}
({\mathcal P})\hspace{10mm}
 \min_{x\in \mathcal{H}}f(x) .
\end{equation*}
 We denote by  $\hbox{argmin}_{\mathcal{H}}f$ the set of minimizers of the optimization problem $({\mathcal P})$, which is assumed to be non-empty.
Our work is part of the active research stream that studies the close link between continuous dissipative dynamical systems and optimization algorithms. In general, the implicit temporal discretization of continuous gradient-based dynamics provides proximal algorithms that benefit from similar asymptotic convergence properties, see \cite{PS} for a systematic study in the case of first-order evolution systems, and 
\cite{att8,att9,ACFR,att6,ACR-Pafa-2019,Bot1,Bot2,Bot3} for some recent results concerning second-order evolution equations.
The main object of our study is the  second-order in time differential  equation
\begin{equation*}
{\rm (IGS)}_{\gamma,\beta,b} \hspace{10mm}
\ddot{x}(t)+\gamma(t)\dot{x}(t)+ \beta(t)\nabla^2 f(x(t))\dot{x}(t)+b(t)\nabla f(x(t))=0,
 \end{equation*}
where the coefficients  $\gamma, \beta: [t_0, +\infty[ \rightarrow \R_{+}$ take account of the viscous  and  Hessian-driven damping, respectively, and  $b:  \R_{+} \to \R_{+}$ is a time scale parameter. We take for granted the existence and uniqueness of the solution of the corresponding Cauchy problem with initial conditions $x(t_0)=x_0\in \cH$, $\dot{x}(t_0)=v_0\in \cH$. Assuming that $\nabla f$ is Lipschitz continuous on the bounded sets, and that the coefficients are continuously differentiable, the local existence  follows from the nonautonomous version of the Cauchy-Lipschitz theorem, see \cite[Prop. 6.2.1]{haraux}. The global existence then follows from the energy estimates that will be established in the next section. 
Each of these damping and rescaling terms properly tuned, improves the rate of convergence of the  associated dynamics and algorithms. An original aspect of our work is to combine them in the same dynamic. Let us recall some classical facts.

\subsection{Damped inertial dynamics and optimization} 
The continuous-time perspective   gives a mechanical intuition of the behavior of the trajectories, and a valuable tool to develop a Lyapunov analysis.
 A first important work in this perspective is the heavy ball with friction method of B. Polyak \cite{polyak} 
$$
{\rm (HBF)} \quad \ddot{x}(t)+\gamma\dot{x}(t)+\nabla f(x(t))=0.
$$
It is a simplified model for  a heavy ball (whose  mass has been normalized to one) sliding on the graph of the function $f$ to be minimized, and which asymptotically stops under the action of viscous friction, see \cite{AGR} for further details. In this model, the  viscous friction parameter $\gamma$ is a fixed positive parameter. 
Due to too much friction (at least asymptotically) involved in this process, replacing  the fixed viscous coefficient with a vanishing viscous coefficient (\ie which tends to zero as $t \to +\infty$) gives  Nesterov's famous accelerated gradient method \cite{nestro1} \cite{nestro2}.
The  other two basic ingredients that we will use, namely  time rescaling, and  Hessian-driven damping have a natural interpretation (cinematic and geometric, respectively) in this context. We will come back to these points later.
Precisely, we seek to develop fast first-order methods  based on the temporal discretization  of damped inertial dynamics.
By fast we mean that, for a  general convex function $f$, and for each trajectory of the system, the convergence rate of the values 
$f(x(t)) - \inf_{\cH} f$ which is obtained is optimal (\ie is achieved of nearly achieved in the worst case). 
The importance of simple first-order methods, and in particular gradient-based and proximal algorithms, comes from the applicability of these algorithms  to a wide
range of large-scale  problems arising from machine learning and/or engineering.

\subsubsection{The viscous damping parameter $\gamma (t)$.}\label{sec:vanishing_damping}

A significant number of recent studies have focused on the  case $\gamma(t)=\frac{\alpha}{t}$,  $\beta=0$ (without Hessian-driven damping), and   $b=1$ (without time rescaling), that is 
\begin{equation*}
{\rm (AVD)}_{\alpha} \hspace{10mm}
\ddot{x}(t)+\frac{\alpha}{t}\dot{x}(t)+\nabla f(x(t))=0.
 \end{equation*}
 This dynamic involves an   Asymptotically Vanishing Damping coefficient (hence the terminology), a key property to obtain fast convergence for a general convex function $f$. In \cite{boyd},  
 Su, Boyd and Cand\`es  showed that for $\alpha=3$ the above system can be seen as a continuous
version of the accelerated gradient method of Nesterov  \cite{nestro1,nestro2} with 
$ f(x(t))-\min_{\mathcal{H}}f = \mathcal O( \frac{1}{t^{2}})$ as $t \to +\infty$. 
The importance of the  parameter $\alpha$ was put to the fore by
Attouch, Chbani, Peypouquet and Redont \cite{redon} and May \cite{may}. They showed that, for $\alpha>3$, one can pass from capital $\mathcal O$ estimates to small $o$. Moreover, when $\alpha >3$, each trajectory converges weakly, and its limit belongs to $\argmin f$\footnote{Recall that for $\alpha =3$ the convergence of the trajectories is an open question}.
Recent research considered the case of a general damping coefficient $\gamma(\cdot)$ (see \cite{cabot1,att2}), thus providing a complete picture of the convergence rates for ${\rm (AVD)}_{\alpha}$:  $ f(x(t))-\min_{\mathcal{H}}f = \mathcal O( 1/t^{2})$ when  $\alpha \geq 3$, and  $f(x(t))-\min_{\mathcal{H}}f = \mathcal O\left( 1/t^{\frac{2\alpha}{3}}\right)$ when  $\alpha\leq 3$, see  \cite{att2,att1} and Apidopoulos, Aujol and Dossal \cite{AAD}.

\subsubsection{The Hessian-driven damping parameter $\beta(t)$.}
 The inertial system 
\begin{equation*}
{\rm \mbox{(DIN)}}_{\gamma,\beta} \qquad \ddot{x}(t) + \gamma \dot{x}(t) + \beta \nabla^2 f (x(t)) \dot{x}(t)  + \nabla f (x(t)) = 0,
\end{equation*}
was  introduced by Alvarez, Attouch, Bolte, and Redont in \cite{AABR}. In line with (HBF), it contains a \textit{fixed} positive friction coefficient $\gamma$. As a main property, the introduction of the Hessian-driven damping makes it possible to neutralize the transversal oscillations likely to occur with (HBF), as observed in \cite{AABR}. The need to take a geometric damping adapted to $f$ had already been observed by Alvarez \cite{Alvarez} who considered the inertial system
\[
\ddot{x}(t) + \mathcal D \dot{x}(t) + \nabla f (x(t)) = 0 ,
\] 
where $\mathcal D : \cH \to \cH$ is a linear positive definite anisotropic operator. But still this damping operator is fixed. For a general convex function, the Hessian-driven damping in $\mbox{\rm (DIN)}_{\gamma,\beta}$ performs a similar operation in a closed-loop adaptive way. (DIN) stands shortly for Dynamical Inertial Newton, and  refers to the link with the Levenberg-Marquardt regularization of the continuous Newton method.
Recent studies have been devoted to the study of the inertial dynamic 
\[
\ddot{x}(t) + \frac{\alpha}{t} \dot{x}(t)+  \beta \nabla^2 f (x(t)) \dot{x}(t)  + \nabla f (x(t)) = 0 ,
\]  
which combines asymptotic vanishing damping  with Hessian-driven damping \cite{APR}.

 \subsubsection{The time rescaling parameter b(t).}
 In the context of non-autonomous dissipative dynamic systems, reparameterization in time is a simple and universal means to accelerate the convergence of trajectories.
 This is where the  coefficient $b(t)$ comes in as a factor of $\nabla f(x(t))$.
In \cite{att6} \cite{ACR-Pafa-2019},  in the case of  general coefficients $\gamma(\cdot)$ and $b(\cdot)$  without the Hessian damping, the authors  made in-depth study. In the case $\gamma(t)=\frac{\alpha}{t}$,  they proved    that  under appropriate conditions on $\alpha$ and $b(\cdot)$,  $f(x(t))-\min_{\mathcal{H}}f = \mathcal O( \frac{1}{t^{2}b(t)})$.
Hence a clear improvement of the convergence rate by taking $b(t) \to +\infty$ as $t\to+\infty$.

\subsection{From damped inertial dynamics  to proximal-gradient inertial algorithms}

Let's review  some classical facts concerning the close link between continuous dissipative inertial dynamic systems and the corresponding algorithms obtained by temporal discretization.
Let us insist on the fact that, when the temporal scaling $b(t) \to +\infty$ as $t\to+\infty$, the transposition of the results to the discrete case naturally leads to consider an implicit temporal discretization, \ie inertial proximal  algorithms.
The reason is that, since  $b(t)$ is in front of the gradient, the application of the gradient descent lemma would require  taking a step size that tends to zero.
On the other hand, the corresponding proximal algorithms  involve a proximal coefficient which tends to infinity (large step proximal algorithms).

\subsubsection{The case without the Hessian-driven damping}
The implicit discretization of ${\rm (IGS)}_{\gamma,0,b}$  gives the  Inertial Proximal algorithm
\begin{equation*}
{\rm (IP)}_{\alpha_{k},\lambda_{k}}\hspace{10mm}
\left\{\begin{array}{l} y_{k}=x_{k}+\alpha_{k}(x_{k}-x_{k-1}) \vspace{2mm}\\ 
 x _{k+1} = \hbox{prox}_{ \lambda_{k}f} ( y_{k} ) \end{array}\right.
\end{equation*}
where $\alpha_{k}$ is non-negative and $\lambda_{k}$ is positive. Recall that for any $\lambda >0$, the proximity operator $\hbox{prox}_{\lambda f}: \cH \to \cH$  is defined by the following formula: for every $x\in \cH$
$$
\hbox{prox}_{\lambda f} (x):=\hbox{argmin}_{\xi\in \cH} \left\lbrace f( \xi ) + \frac1{2\lambda} \| x - \xi \|^2 \right\rbrace.$$ 
Equivalently, $\hbox{prox}_{\lambda f}$ is the resolvent of index $\lambda$ of the maximally monotone operator $\partial f$. 
When passing to the implicit discrete case,
we can take $f: \mathcal{H} \rightarrow  \R \cup \{+\infty\} $ a convex lower semicontinuous and proper function.
Let us list some of the main results concerning the convergence properties of the algorithm ${\rm (IP)}_{\alpha_{k},\lambda_{k}}$:

\smallskip

$\bullet$ 1.  Case  $\lambda_{k}  \equiv \lambda >0$  and $\alpha_k=1-\frac{\alpha}{k}$. When $\alpha=3$,  the  ${\rm (IP)}_{1-3/k,\lambda}$ algorithm has a similar structure to the original  Nesterov accelerated gradient algorithm\cite{nestro1}, just replace the gradient step with a proximal step. Passing from the gradient to the proximal step was carried out by G\"uler \cite{Guler1,Guler2}, then by Beck and Teboulle \cite{BeckTeboulle} for structured optimization.  
A decisive step was taken by   Attouch and Peypouquet in \cite{att7}  proving  that,   when $\alpha>3$, $f(x_{k})-\min_{\mathcal{H}}f = o\left(\frac{1}{k^{2}}\right)$. The subcritical case $\alpha< 3$ was  examined by Apidopoulos, Aujol, and Dossal \cite{AAD} and Attouch, Chbani, and Riahi \cite{att1} with the rate of convergence rate of values $f(x_{k})-\min_{\mathcal{H}}f = \mathcal O\left(\frac{1}{k^{\frac{2\alpha}{3}}}\right)$. 

\smallskip

$\bullet$ 2.  For a general $\alpha_k$,  the convergence properties of ${\rm (IP)}_{\alpha_{k},\lambda}$ were analyzed by  Attouch and Cabot \cite{att8},  then by Attouch, Cabot, Chbani, and Riahi \cite{att9},  in  the presence of perturbations. 
The convergence rates are then expressed    using  the sequence  $ (t_k )$ which is  linked to $(\alpha_k)$  by the formula 
$t_k:= 1+\sum_{i=k}^{+\infty}\prod_{j=k}^i\alpha_{j} $.
Under growth conditions on $t_k$, it is proved  that  $f(x_{k})-\min_{\mathcal{H}}f = \mathcal O(\frac{1}{t_{k}^{2}})$. This last results covers the special case $\alpha_k=1-\frac{\alpha}{k}$  when $\alpha \geq 3$.

\smallskip

$\bullet$ 3. For a general $\lambda_k$,  Attouch, Chbani, and Riahi first considered in \cite{att6} the case   $\alpha_k=1-\frac{\alpha}{k}$. They proved  that under a growth condition on $\lambda_k $, we have the estimate  $f(x_{k})-\min_{\mathcal{H}}f = \mathcal O(\frac{1}{k^{2}\lambda_{k}})$.
 This result is an improvement of the one discussed previously in \cite{att7}, because when $ \lambda_k = k^{\delta} $ with $ 0 <\delta <\alpha-3 $, we pass from $ \mathcal O (\frac {1} {k^{2}}) $ to $ \mathcal O (\frac {1} {k^{2+ \delta}}) $.
   Recently, in \cite{att11} the authors   analyzed   the algorithm ${\rm (IP)}_{\alpha_{k},\lambda_{k}}$ for  general $\alpha_k$ and $\lambda_k$. By including the expression  of $t_k$  previously used in \cite{att8,att9}, they proved that $f(x_{k})-\min_{\mathcal{H}}f = \mathcal O\left(1/ t_{k}^{2} \lambda_{k-1}\right)$ under certain conditions on $\lambda_k$ and $\alpha_{k}$. They obtained $f(x_{k})-\min_{\mathcal{H}}f = 
   o\left(1/t_{k}^{2}\lambda_{k}\right)$, which gives a global view of of the convergence rate with small $o$, encompassing 
\cite{att8,att11}.
   
\subsubsection{The case with the Hessian-driven damping}

Recent studies have been devoted to  the inertial dynamic 
\[
\ddot{x}(t) + \frac{\alpha}{t} \dot{x}(t)+  \beta \nabla^2 f (x(t)) \dot{x}(t)  + \nabla f (x(t)) = 0 ,
\]  
which combines asymptotic vanishing viscous damping  with Hessian-driven damping.
The corresponding algorithms involve a correcting term in the Nesterov accelerated gradient method which reduces the oscillatory aspects, see  Attouch-Peypouquet-Redont \cite{APR}, Attouch-Chbani-Fadili-Riahi \cite{ACFR},
 Shi-Du-Jordan-Su   \cite{SDJS}. The case of  monotone inclusions has been  considered by   Attouch and L\'aszl\'o \cite{AL}.
 
 \smallskip
 
 \subsection{Contents}
The paper is organized as follows. In section \ref{sec:general},
we  develop a new Lyapunov analysis for the continuous dynamic
${\rm (IGS)}_{\gamma,\beta,b}$. In Theorem \ref{th1aa}, we  provide a system of conditions on the damping parameters
$ \gamma (\cdot) $ and $ \beta (\cdot) $, and on the temporal scaling parameter $ b (\cdot) $ giving fast convergence of the values.
Then, in sections \ref{Gamma} and \ref{p_gamma}, 
we present two different types of growth conditions for the damping and temporal scaling parameters, respectively based on the  functions 
 ${\rm \Gamma}_{\gamma}$ and $p_{\gamma}$, and which satisfy the conditions of  Theorem \ref{th1aa}.
In doing so, we encompass most existing results and provide new results, including  linear convergence rates without assuming strong convexity.
This will also allow us to explain the choice of certain coefficients in the associated algorithms, questions which have remained mysterious and only justified by the  simplification of often complicated calculations.
In section \ref{illustrations}, we  specialize our results to certain model situations and give numerical illustrations.
Finally, we conclude the paper by highlighting its original aspects.

 \section{Convergence rate of the values. General abstract result}
 \label{sec:general}
  We will establish a general result concerning the  convergence rate of the values verified by the solution trajectories $x(\cdot)$   of the second-order evolution equation 
 \begin{equation*}
{\rm (IGS)_{\gamma,\beta,\textit{b}}}\hspace{10mm}
\ddot{x}(t)+\gamma(t)\dot{x}(t)+\beta(t)\nabla^2 f(x(t))\dot{x}(t) +b(t)\nabla f(x(t))=0.
 \end{equation*}
The variable parameters  $\gamma(\cdot)$, $\beta(\cdot)$ and $b(\cdot) $ take into account  the damping, and temporal rescaling effects. They are assumed to be continuously differentiable.
To analyze the asymptotic behavior of the solutions trajectories of the evolution system ${\rm (IGS)_{\gamma,\beta,\textit{b}}}$, we will use Lyapunov's analysis. It is a classic and powerful tool which consists in building an associated energy-like function which decreases along the trajectories.
The determination of such a Lyapunov function is in general a delicate problem. Based on previous works, we know the global structure of such a Lyapunov function. It is a weighted sum of the potential, kinetic and anchor functions.  We will introduce coefficients in this function that are a priori unknown, and which  will be identified during the calculation to verify the property of decay.
Our approach takes advantage of the technics recently developed in \cite{cabot1},  \cite{APR}, \cite{ACR-Pafa-2019}.
\subsection{The general case}
Let $x(\cdot)$ be a solution trajectory of  ${\rm (IGS)_{\gamma,\beta,\textit{b}}}$. 
Given $ z \in \hbox{argmin}_{\cH} f $, we  introduce the  Lyapunov function $t \mapsto \mathcal{E}(t)$ defined by 
\begin{equation}\label{equ1}
\mathcal{E}(t):= c(t)^{2}b(t)\left(f(x(t))-f(z)\right)+\frac{\theta(t)\sigma (t)^{2}}{2}\left\| v(t)\right\|^2 + \frac{\xi (t)}{2} \Vert x (t) -z\Vert^{2} ,
\end{equation}
where $v(t):=x(t) -z+\frac1{ \sigma(t)}\left(\dot{x}(t) +\beta(t) \nabla f(x(t))\right).$

\medskip

\noindent The  four variable coefficients $c(t), \theta(t), \sigma (t), \xi(t)$ will be adjusted during the calculation.
According to the classical derivation chain rule, we obtain 
\begin{eqnarray*}
\frac{d}{dt} \mathcal{E}(t) & = & \frac{d}{dt}\left( c^{2}(t)b(t)\right) \left(f(x(t))-f(z)\right)+c(t)^{2}b(t)\langle\nabla f(x(t)),\dot{x}(t)\rangle \\ &+& \frac12\frac{d}{dt}(\theta(t) \sigma^2 (t))\|v (t)\|^2 
 +\theta(t) \sigma^2 (t)\langle\dot{v}(t),v (t)\rangle  \\
 &+& \demi\dot\xi (t)\|x (t)-z\|^2 +\xi (t)\langle\dot{x} (t) , x(t)-z\rangle.
\end{eqnarray*}
From now, without ambiguity, to shorten formulas, we omit the variable $t$.\\
According to the definition of $v$, and the  equation ${\rm (IGS)_{\gamma,\beta,\textit{b}}}$, we have 
\begin{eqnarray*}
\dot{v } &=& \dot{x} -\frac{\dot{\sigma}}{ \sigma^2} \left(\dot{x} + \beta \nabla f(x)\right) + 
\frac1{ \sigma}\frac{d}{dt}\left(\dot{x} + \beta \nabla f(x)\right) \\ 
	&=& \dot{x} -\frac{\dot{\sigma}}{ \sigma^2}\left(\dot{x} + \beta \nabla f(x)\right) + 
\frac1{ \sigma}\left(\ddot{x} + \beta \nabla^2 f(x)\dot x + \dot \beta \nabla f(x) \right) \\ 
	&=& \dot{x} -\frac{\dot{\sigma}}{ \sigma^2}\left(\dot{x} + \beta \nabla f(x)\right) + 
\frac1{ \sigma}\left(-\gamma\dot{x} - b\nabla f(x)  + \dot \beta \nabla f(x) \right) \\ 
	&=& \left( 1 -\frac{\dot{\sigma}}{ \sigma^2} -  \frac{\gamma}{ \sigma}\right)\dot{x} + 
	 \left(\frac{\dot \beta}{ \sigma}  -\frac{\dot{\sigma}\beta}{ \sigma^2} - \frac{b}{ \sigma} \right)\nabla f(x).
\end{eqnarray*}
Therefore, 
\begin{eqnarray*}
\langle\dot{v},v\rangle &=&  \left\langle \left( 1 -\frac{\dot{\sigma}}{ \sigma^2} -  \frac{\gamma}{ \sigma}\right)\dot{x} + 
	 \left(\frac{\dot \beta}{ \sigma}  -\frac{\dot{\sigma}\beta}{ \sigma^2} - \frac{b}{ \sigma} \right)\nabla f(x)\;,\; x -z+\frac1{ \sigma}\left(\dot{x} + \beta \nabla f(x)\right)\right\rangle \\ 
	 &=& \left( 1 -\frac{\dot{\sigma}}{ \sigma^2} -  \frac{\gamma}{ \sigma}\right)  \left\langle \dot{x}\;,\; x -z\right\rangle +  \left( 1 -\frac{\dot{\sigma}}{ \sigma^2} -  \frac{\gamma}{ \sigma}\right)\frac1{ \sigma}\| \dot{x}\|^2\\ 
	 &+&  \left(\left( 1 -\frac{\dot{\sigma}}{ \sigma^2} -  \frac{\gamma}{ \sigma}\right) \frac{\beta}{ \sigma} +  \left(\frac{\dot \beta}{ \sigma}  -\frac{\dot{\sigma}\beta}{ \sigma^2} - \frac{b}{ \sigma} \right)\frac{1}{ \sigma}\right) \left\langle  \nabla f(x)\;,\;  \dot{x}\right\rangle \\ 
	 &+&  \left(\frac{\dot \beta}{ \sigma}  -\frac{\dot{\sigma}\beta}{ \sigma^2} - \frac{b}{ \sigma} \right)\left\langle  \nabla f(x)\;,\;  x-z\right\rangle +  \left(\frac{\dot \beta}{ \sigma}  -\frac{\dot{\sigma}\beta}{ \sigma^2} - \frac{b}{ \sigma} \right)\frac{\beta}{ \sigma}\|\nabla f(x)\|^2.
\end{eqnarray*}
According to the definition of $v(t)$, after developing 
$ \|v(t)\|^2$, we get
\begin{eqnarray*}
 \|v\|^2 &=&  \|x -z\|^2+\frac1{ \sigma^2}\left(\|\dot{x}\|^2 + \beta^2 \|\nabla f(x)\|^2\right) +\frac2{\sigma} \left\langle \dot{x}\;,\; x -z\right\rangle\\ 
 && + \frac{2\beta}{\sigma}\left\langle  \nabla f(x)\;,\;  x-z\right\rangle +\frac{2\beta}{\sigma^2}\left\langle  \nabla f(x)\;,\;  \dot{x}\right\rangle. 
 \end{eqnarray*}
Collecting the above results, we obtain
\begin{eqnarray*}
\frac{d}{dt} \mathcal{E}(t) & = & \frac{d}{dt}\left( c^{2}b\right) \left(f(x)-f(z)\right)+c^{2}b\langle\nabla f(x),\dot{x}\rangle  + \demi\dot\xi \|x -z\|^2 +\xi \langle\dot{x} , x-z\rangle \\
&&+
\frac12\frac{d}{dt}(\theta \sigma^2 ) \Big( \|x -z\|^2+\frac1{ \sigma^2}\left(\|\dot{x}\|^2 + \beta^2 \|\nabla f(x)\|^2\right) +\frac2{\sigma} \left\langle \dot{x}\;,\; x -z\right\rangle \Big)\\
 && +\frac12\frac{d}{dt}(\theta \sigma^2 ) \Big( \frac{2\beta}{\sigma}\left\langle  \nabla f(x)\;,\;  x-z\right\rangle +\frac{2\beta}{\sigma^2}\left\langle  \nabla f(x)\;,\;  \dot{x}\right\rangle \Big)\\
 && +\theta \sigma^{2}\Big( \left( 1 -\frac{\dot{\sigma}}{ \sigma^2} -  \frac{\gamma}{ \sigma}\right)  \left\langle \dot{x}\;,\; x -z\right\rangle +  \left( 1 -\frac{\dot{\sigma}}{ \sigma^2} -  \frac{\gamma}{ \sigma}\right)\frac1{ \sigma}\| \dot{x}\|^2 \Big) \\
 && +\theta \sigma^{2} \left(\left( 1 -\frac{\dot{\sigma}}{ \sigma^2} -  \frac{\gamma}{ \sigma}\right) \frac{\beta}{ \sigma} +  \left(\frac{\dot \beta}{ \sigma}  -\frac{\dot{\sigma}\beta}{ \sigma^2} - \frac{b}{ \sigma} \right)\frac{1}{ \sigma}\right) \left\langle  \nabla f(x)\;,\;  \dot{x}\right\rangle \\
 && +\theta \sigma^{2}\Big( \left(\frac{\dot \beta}{ \sigma}  -\frac{\dot{\sigma}\beta}{ \sigma^2} - \frac{b}{ \sigma} \right)\left\langle  \nabla f(x)\;,\;  x-z\right\rangle +  \left(\frac{\dot \beta}{ \sigma}  -\frac{\dot{\sigma}\beta}{ \sigma^2} - \frac{b}{ \sigma} \right)\frac{\beta}{ \sigma}\|\nabla f(x)\|^2 \Big).
\end{eqnarray*}
In the second member of the above formula, let us examine the terms
that contain 
$\left\langle  \nabla f(x)\;,\;  x-z\right\rangle$. By grouping these terms, we obtain the following expression
$$
\Big(\frac{\beta}{\sigma}\frac{d}{dt}(\theta \sigma^2 ) +  \theta \sigma^{2} \left(\frac{\dot \beta}{ \sigma}  -\frac{\dot{\sigma}\beta}{ \sigma^2} - \frac{b}{ \sigma} \right)    \Big) \left\langle  \nabla f(x)\;,\;  x-z\right\rangle.
$$
To majorize it,  we use the convex subgradient inequality 
$
 \langle \nabla f(x),x-z \rangle \geq  f(x)-f(z),
$
and we make a first hypothesis  
$
\frac{\beta}{\sigma}\frac{d}{dt}(\theta \sigma^2 ) +  \theta \sigma^{2} \left(\frac{\dot \beta}{ \sigma}  -\frac{\dot{\sigma}\beta}{ \sigma^2} - \frac{b}{ \sigma} \right)  \leq 0.
$
Therefore, 
\begin{eqnarray}
\frac{d}{dt} \mathcal{E}(t)   
	&& \leq   \left[ \frac{d}{dt}\left( c^{2}b\right) +\frac{\beta}{\sigma}\frac{d}{dt}(\theta \sigma^2 ) +  \theta \sigma^{2} \left(\frac{\dot \beta}{ \sigma}  -\frac{\dot{\sigma}\beta}{ \sigma^2} - \frac{b}{ \sigma} \right)  \right] \left(f(x)-f(z)\right) \nonumber \\ 
	 &&+  \left[c^{2}b + \frac{\beta}{\sigma^2}\frac{d}{dt}(\theta \sigma^2 )  + \theta \sigma\left( \left( 1 -\frac{\dot{\sigma}}{ \sigma^2} -  \frac{\gamma}{ \sigma}\right) \beta +  \left(\frac{\dot \beta}{ \sigma}  -\frac{\dot{\sigma}\beta}{ \sigma^2} - \frac{b}{ \sigma} \right)\right)\right] \left\langle  \nabla f(x)\;,\;  \dot{x}\right\rangle \nonumber\\
	 &&+   \left[  \frac{1}{\sigma}\frac{d}{dt}(\theta \sigma^2 ) +\theta  \sigma^{2}\left( 1 -\frac{\dot{\sigma}}{ \sigma^2} -  \frac{\gamma}{ \sigma}\right) + \xi\right]   \left\langle \dot{x}\;,\; x -z\right\rangle  \nonumber
	 \\ 
	&&+ \demi\left[ \frac{d}{dt}(\theta \sigma^2 )  +   \dot\xi \right]\|x-z\|^2+\left[ \frac{1}{2\sigma^2}\frac{d}{dt}(\theta \sigma^2 )   + \theta \sigma \left( 1 -\frac{\dot{\sigma}}{ \sigma^2} -  \frac{\gamma}{ \sigma}\right)  \right]\|\dot{x}\|^{2} \nonumber \\  
	&&+ \left[ \frac{\beta^2}{2\sigma^2}\frac{d}{dt}(\theta \sigma^2 ) +  \theta \sigma \beta \left(\frac{\dot \beta}{ \sigma}  -\frac{\dot{\sigma}\beta}{ \sigma^2} - \frac{b}{ \sigma}\right) \right]\|\nabla f(x)\|^2. \label{equ2}
\end{eqnarray}
To get $\frac{d}{dt} \mathcal{E}(t)\leq 0$, we are led to make the following assumptions:
\begin{eqnarray*}
(i) &\hspace{5mm}& \frac{\beta}{\sigma}\frac{d}{dt}(\theta \sigma^2 ) +  \theta \sigma^{2} \left(\frac{\dot \beta}{ \sigma}  -\frac{\dot{\sigma}\beta}{ \sigma^2} - \frac{b}{ \sigma} \right)  \leq 0\\
(ii) &\hspace{5mm}& \frac{d}{dt}\left( c^{2}b\right) + \frac{\beta}{\sigma}\frac{d}{dt}(\theta \sigma^2 ) +  \theta \sigma^{2} \left(\frac{\dot \beta}{ \sigma}  -\frac{\dot{\sigma}\beta}{ \sigma^2} - \frac{b}{ \sigma} \right)  \leq 0,\\ 
(iii) && c^{2}b + \frac{\beta}{\sigma^2}\frac{d}{dt}(\theta \sigma^2 )  + \theta \sigma\left( \left( 1 -\frac{\dot{\sigma}}{ \sigma^2} -  \frac{\gamma}{ \sigma}\right) \beta +  \left(\frac{\dot \beta}{ \sigma}  -\frac{\dot{\sigma}\beta}{ \sigma^2} - \frac{b}{ \sigma} \right)\right) = 0, \\
(iv) &&    \frac{1}{\sigma}\frac{d}{dt}(\theta \sigma^2 )  +\theta  \sigma^{2}\left( 1 -\frac{\dot{\sigma}}{ \sigma^2} -  \frac{\gamma}{ \sigma}\right) + \xi =0,\\
(v) &&  \frac{d}{dt}(\theta \sigma^2 )  +   \dot\xi\leq 0,\\
(vi) &&  \frac{1}{2\sigma^2}\frac{d}{dt}(\theta \sigma^2 )   + \theta \sigma \left( 1 -\frac{\dot{\sigma}}{ \sigma^2} -  \frac{\gamma}{ \sigma}\right)    \leq 0,\\  
(vii) &&  \frac{\beta^2}{2\sigma^2}\frac{d}{dt}(\theta \sigma^2 ) +  \theta \sigma \beta \left(\frac{\dot \beta}{ \sigma}  -\frac{\dot{\sigma}\beta}{ \sigma^2} - \frac{b}{ \sigma}\right) \leq 0.
\end{eqnarray*}
After simplification, we get the following equivalent system of conditions:
\begin{eqnarray*}
\boxed{
\begin{array}{rcl}
&& {\rm \mbox{A: Lyapunov system of inequalities involving }}\; \; c(t), \theta(t), \sigma (t), \xi(t).  \hspace{1cm}
\\
\hline
 \smallskip\\
 &&(i) \hspace{5mm}  \frac{d}{dt}( \beta\theta\sigma)     - \theta b \sigma \leq 0 \vspace{3mm} \\
&&(ii) \hspace{5mm}\frac{d}{dt}\left( c^{2}b + \beta\theta\sigma \right)     - \theta b \sigma \leq 0,\vspace{3mm} \\  
&&(iii) \hspace{5mm}  b(c^{2}-\theta )  + \beta \theta  (\sigma-\gamma) + \frac{d}{dt}( \beta\theta) = 0, \vspace{3mm} \\
&&(iv) \hspace{5mm}    \frac{d}{dt}(\theta\sigma) + \theta\sigma \left( \sigma -  \gamma \right) + \xi =0,\vspace{3mm} \\
&&(v)   \hspace{5mm}  \frac{d}{dt}(\theta\sigma^2 +\xi) \leq 0, \vspace{3mm}\\
&&(vi)  \hspace{5mm}  \dot\theta+ 2(\sigma -\gamma)\theta   \leq 0, 
\vspace{3mm}\\  
&&(vii)\hspace{5mm}  \beta\Big( \beta \dot\theta + 2 \left(\dot \beta -b\right) \theta \Big) \leq 0. 
\end{array}
}
\end{eqnarray*}
%
%
Let's simplify this system by eliminating the variable $\xi$.
From $(iv)$ we get $\xi =- \frac{d}{dt}(\theta\sigma) - \theta\sigma \left( \sigma -  \gamma \right) $, that we
replace in $(v)$, and recall that $\xi$ is prescribed to be nonnegative.
Now observe that the unkown function $c$ can also be eliminated. Indeed, it enters the above system via the variable $bc^2$, which according to $(iii)$ is equal to 
$bc^2 = b\theta -\beta \theta  (\sigma-\gamma)- \frac{d}{dt}( \beta\theta). $
Replacing in $(ii)$, which is the only other equation involving $bc^2$, we   obtain the equivalent system involving only the variables $\theta(t), \sigma (t)$.
\begin{eqnarray*}
\boxed{
\begin{array}{rcl}
&& { \rm \mbox{B: Lyapunov system of inequalities involving the variables:}}\; \; \theta(t), \sigma (t) \hspace{5mm} 
\vspace{2mm} \\
\hline
\vspace{1mm}\\
&& (i) \hspace{5mm}  \frac{d}{dt}( \beta\theta\sigma)     - \theta b \sigma \leq 0, \vspace{2mm} \\
&&(ii) \hspace{5mm} \frac{d}{dt} \left( b\theta + \beta \theta \gamma \right) - \frac{d^2}{dt^2}( \beta\theta)   - \theta b \sigma \leq 0,\vspace{2mm}\\  
&&(iii)  \hspace{5mm}   b\theta -\beta \theta  (\sigma-\gamma)- \frac{d}{dt}( \beta\theta) \geq 0, \vspace{2mm}\\
&&(iv)  \hspace{5mm}    \frac{d}{dt}(\theta\sigma) + \theta\sigma \left( \sigma -  \gamma \right) \leq 0, \vspace{2mm}\\
&&(v)  \hspace{5mm}   \frac{d}{dt}\left(- \frac{d}{dt}(\theta\sigma) + \theta\sigma \gamma \right) \leq 0, \vspace{2mm}\\
&&(vi)  \hspace{5mm}   \dot\theta+ 2(\sigma -\gamma)\theta   \leq 0, \vspace{2mm}\\  
&&(vii)  \hspace{5mm}   \beta\Big( \beta \dot\theta + 2 \left(\dot \beta -b\right) \theta \Big) \leq 0 \vspace{2mm}.
\end{array}
}
\end{eqnarray*}
Then, the variables $\xi$ and $c$ are obtained by using the formulas
\begin{eqnarray*}
&& \xi =- \frac{d}{dt}(\theta\sigma) - \theta\sigma \left( \sigma -  \gamma \right) \\
&& bc^2 = b\theta -\beta \theta  (\sigma-\gamma)- \frac{d}{dt}( \beta\theta).
\end{eqnarray*}
Thus, under the above conditions,  the function $\mathcal{E}(\cdot)$ is  nonnegative and
nonincreasing.
Therefore,  for every $t\geq t_{0}$, $\mathcal{E}(t)\leq \mathcal{E}(t_{0}),$
which implies that 
$$c^{2}(t)b(t)f(x(t))-\min_\mathcal{H}f)\leq \mathcal{E}(t_{0}).$$
Therefore, as $t \to +\infty$
 $$f(x(t))-\min_\mathcal{H}f= \mathcal O\left(\frac1{c^{2}(t)b(t)}\right).$$
 Moreover, by integrating (\ref{equ2}) we obtain the following integral estimates:
 
 \noindent a) On the values:
 $$
 \int_{t_0}^{+\infty} \Big(\theta(t) b (t)\sigma (t) - \frac{d}{dt}\left( c^{2}(t)b(t) + \beta(t)\theta(t)\sigma(t) \right)   \Big)  \left(f(x(t))- 
 \inf_{\cH} f \right) dt < +\infty;
 $$
 where we use the equality:
 
\smallskip
 
  $-\Big( \frac{d}{dt}\left( c^{2}b\right) +\frac{\beta}{\sigma}\frac{d}{dt}(\theta \sigma^2 ) +  \theta \sigma^{2} \left(\frac{\dot \beta}{ \sigma}  -\frac{\dot{\sigma}\beta}{ \sigma^2} - \frac{b}{ \sigma} \right) \Big) = \theta b \sigma-\frac{d}{dt}\left( c^{2}b + \beta\theta\sigma \right)   $
  
\smallskip

\noindent   and the fact that, according to $(ii)$, this quantity is nonnegative.

\medskip 
 
\noindent  b) On the norm of the gradients:
 $$
 \int_{t_0}^{+\infty} q(t)\|\nabla f(x(t))\|^2dt < +\infty.
 $$
 where $q$ is the  nonnegative weight function defined by

\begin{eqnarray}\label{int_est_q}
q(t) &\eqdef & \theta (t)\beta (t)\left(  \frac{\dot{\sigma} (t)\beta (t)}{ \sigma (t)} + b (t) -\dot \beta (t)\right)- \frac{\beta^2(t)}{2\sigma^2(t)}\frac{d}{dt}(\theta \sigma^2 )(t) \nonumber\\
	&=&   b(t)\theta (t)\beta (t) -  \demi \frac{d}{dt}(\theta \beta ^2 )(t).
\end{eqnarray}
We can now state the following Theorem, which summarizes the above results.

 \begin{theorem}\label{th1aa}
  Let $f: \mathcal{H} \rightarrow \R $  be  a convex differentiable  function with $\argmin_{\mathcal{H}}f \neq \emptyset$. \\
 Let $x(\cdot)$ be a solution trajectory of 
  \begin{equation*}
{\rm (IGS)_{\gamma,\beta,\textit{b}}}\hspace{10mm}
\ddot{x}(t)+\gamma(t)\dot{x}(t)+\beta(t)\nabla^2 f(x(t))\dot{x}(t) +b(t)\nabla f(x(t))=0.
 \end{equation*}
  Suppose  that  $\gamma(\cdot)$, $\beta(\cdot)$, and $b(\cdot)$, are  $\cC^1$  functions on $[t_0,+\infty [$ such that there exists auxiliary functions  $c(t), \theta(t), \sigma (t), \xi(t)$  that satisfy the conditions $(i)-(vii)$ above.
 Set
 \begin{equation}\label{equ1b}
\mathcal{E}(t):= c(t)^{2}b(t)\left(f(x(t))-f(z)\right)+\frac{\theta(t)\sigma (t)^{2}}{2}\left\| v(t)\right\|^2 + \frac{\xi (t)}{2} \Vert x (t) -z\Vert^{2} ,
\end{equation}
with $z \in \argmin_{\mathcal{H}}f$ and $v(t)= x(t) -z+\frac1{ \sigma(t)}\left(\dot{x}(t) +\beta(t) \nabla f(x(t))\right)$.

\smallskip

\noindent  Then, $t \mapsto \mathcal{E}(t)$ is a nonincreasing function. As a consequence, for all $t \geq t_0$ ,
\begin{eqnarray}
&&(i) \,  f(x(t))-\min_\cH f \leq \frac {\cE (t_0)}{c^{2}(t)b(t)};
\label{cont-rap1} \\
&&
(ii) \int_{t_0}^{+\infty} \Big(\theta(t) b (t)\sigma (t) - \frac{d}{dt}\big( c^{2}b + \beta\theta\sigma \big)(t)   \Big)  \left(f(x(t))- \inf_{\cH} f \right) dt < +\infty;
\label{cont-rap2b}  \\
&&
(iii) \,\int_{t_0}^{+\infty} \left(b(t)\theta (t)\beta (t) - \demi \frac{d}{dt}\big(\theta \beta ^2 \big)(t) \right)\|\nabla f(x(t))\|^2 dt < +\infty.
\label{cont-rap2} 
\end{eqnarray}
\end{theorem}

\subsection{Solving system $(i)-(vii)$}
The system of inequalities $(i)-(vii)$ of Theorem \ref{th1aa}  may seem complicated at first glance.
Indeed, we will see that it  simplifies notably in the classical situations. Moreover, it makes it possible to unify the existing results, and discover new interesting cases.
We will present two different types of solutions to this system, respectively based on the following functions:
\begin{equation}\label{basic_relation_1a}
p_{\gamma}(t)=\exp\left({ \displaystyle{\int_{t_0}^t \gamma (u)\, du}}\right),
\end{equation}
and 
\begin{equation}\label{basic_relation_1b}
{\rm \Gamma}_{\gamma} (t)=p_{\gamma}(t)\int_{t}^{+\infty}\frac{du}{p_{\gamma}(u)} .
\end{equation} 
The use of ${\rm \Gamma}_{\gamma} $ has been considered in a series of articles that we will retrieve as a special case of our approach, see \cite{cabot1}, \cite{att8}, \cite{att2}, \cite{ACR-Pafa-2019}. Using  $p_{\gamma}$ will lead to new results, see section \ref{p_gamma}.

\section{Results based on the function ${\rm \Gamma}_{\gamma}$}
\label{Gamma}
In this section, we will systematically assume that condition $(H_0)$ is satisfied.
$$\int_{t_0}^{+\infty}\frac{ds}{p(s)}<+\infty.\leqno (H_0)$$
Under  $(H_0)$, the function ${\rm \Gamma}_{\gamma}(\cdot)$ is well defined. 
It can be equally defined as the solution of the linear non autonomous  differential equation
\begin{equation}\label{basic_relation_2}
\dot {{\rm \Gamma}}_{\gamma} (t) -\gamma (t) {\rm \Gamma}_{\gamma} (t) +1=0,
\end{equation}
which satisfies the limit condition $\lim_{t \to +\infty} 
\frac{{\rm \Gamma}_{\gamma} (t)}{p_{\gamma}(t)}=0$. 

\subsection{The case without the Hessian, \ie $\beta \equiv 0$}
The dynamic writes
 \begin{equation*}
{\rm (IGS)_{\gamma,0,\textit{b}}}\hspace{10mm}
\ddot{x}(t)+\gamma(t)\dot{x}(t) +b(t)\nabla f(x(t))=0.
 \end{equation*}
To solve the system $(i)-(vii)$ of Theorem \ref{th1aa}, we choose 
$$\xi \equiv 0,\, \, c(t)= {\rm \Gamma}_{\gamma} (t), \, \,  \sigma (t)=\frac1{{\rm \Gamma}_{\gamma} (t)}, \, \, \theta (t)=\Gamma_{\gamma} (t)^{2}.$$
According to (\ref{basic_relation_2}), we can easily verify  that  conditions $(i),(iii)-(vii)$  are satisfied, and $(ii)$ becomes
$$
 \frac{d}{dt}\left( {\rm \Gamma}_{\gamma}(t)^2 b(t)\right) -  {\rm \Gamma}_{\gamma}(t) b (t) \leq 0.
$$
After dividing by ${\rm \Gamma}_{\gamma}(t)$, and using (\ref{basic_relation_2}), we obtain the condition
$$
0\geq    {\rm \Gamma}_{\gamma}(t) \dot b (t)-(3-2\gamma (t) {\rm \Gamma}_{\gamma}(t))b(t) .
$$
This leads to the following result obtained by Attouch, Chbani and Riahi in \cite{ACR-Pafa-2019}.
\begin{theorem}\label{th1ab}{\rm \cite[Theorem 2.1]{ACR-Pafa-2019}}
  Suppose that    for all $t \geq t_0$ 
  \begin{equation}\label{basic_growth_1}
{\rm \Gamma}_{\gamma}  (t) \dot{b}(t)  \leq   b (t) \left( 3-2\gamma(t) {\rm \Gamma}_{\gamma}  (t)\right),
\end{equation}
where ${\rm \Gamma}_{\gamma} $ is defined from $\gamma$ by  {\rm (\ref{basic_relation_1b})}.
Let  $x: [t_0, +\infty[ \rightarrow \mathcal H$ be a solution trajectory of
${\rm (IGS)_{\gamma,0,\textit{b}}}$.
Given $z \in \argmin_\cH f$, set
 \begin{equation}\label{equ1c}
\mathcal{E}(t):= {\rm \Gamma}_{\gamma}^{2} (t) b(t)\left(f(x(t))-f(z)\right)+\frac{1}{2}\left\| x(t) -z + {\rm \Gamma}_{\gamma} (t)\dot{x}(t)\right\|^2 .
\end{equation}
 Then, $t \mapsto \mathcal{E}(t)$ is a nonincreasing function. As a consequence,  as $t\to +\infty$
\begin{equation}\label{cont-rap11}
f (x(t))-\min_\mathcal{H}f= \mathcal O\left(\frac1{{\rm \Gamma}_{\gamma}(t)^{2} b(t)}\right).
\end{equation}
Precisely, for all $t\geq t_0$
\begin{equation}\label{basic-est-008}
f(x(t))-\min_{\cH} f  \leq \frac{C}{  {\rm \Gamma}_{\gamma}(t)^2 b(t)},
\end{equation}
with \,
$
C= {\rm \Gamma}_{\gamma}(t_0)^2 b(t_0)\left(f(x(t_0))-\min_{\cH} f\right)+
d(x(t_0), \argmin f )^2 + {\rm \Gamma}_{\gamma} (t_0)^2 \|\dot x(t_0)\|^2.
$
Moreover,
$$
\int_{t_0}^{+\infty}  {\rm \Gamma}_{\gamma}(t) \Big( b (t) \left( 3-2\gamma(t) {\rm \Gamma}_{\gamma}  (t)\right)- {\rm \Gamma}_{\gamma}  (t) \dot{b}(t)    \Big) (f(x(t))-\min_{\cH} f)dt < +\infty.
$$
\end{theorem}
\begin{remark}
When $b\equiv 1$, condition (\ref{basic_growth_1}) reduces to $\gamma(t) {\rm \Gamma}_{\gamma}  (t) \leq \frac{3}{2}$, introduced in \cite{cabot1}.
\end{remark}

\subsection{Combining Nesterov acceleration with Hessian damping}
Let us specialize our results in the case $\beta (t)>0$, and  
$\gamma (t)=\frac{\alpha}{t}$. We are in the case of a vanishing damping coefficient (\ie $\gamma (t) \to 0$ as $t \to + \infty$). According to Su, Boyd and Cand\`es \cite{boyd}, the case $\alpha =3$ corresponds to a continuous version of the accelerated gradient method of Nesterov. Taking $\alpha >3$ improves in many ways the convergence properties of this dynamic, see section \ref{sec:vanishing_damping}. Here, it is combined with the Hessian-driven damping and temporal rescaling.  This situation was first considered by
Attouch, Chbani, Fadili and Riahi in \cite{ACFR}.
Then the dynamic writes
\begin{equation*}
{\rm (IGS)_{\alpha/t,\beta,\textit{b}}}\hspace{10mm}
\ddot{x}(t)+\frac{\alpha}{t}\dot{x}(t)+\beta(t)\nabla^2 f(x(t))\dot{x}(t) +b(t)\nabla f(x(t))=0.
 \end{equation*}
Elementary calculus gives that  $(H_0)$ is satisfied as soon as $\alpha >1$. In this case, 
$$
{\rm \Gamma}_{\gamma}(t)=  \frac{t}{\alpha -1}.
 $$
 After \cite{ACFR}, let us introduce
the following quantity which  will simplify the formulas: 
\begin{equation}\label{bacic-def}
w(t)\eqdef b(t)-\dot{\beta}(t) -\dfrac{\beta(t)}{t} .
\end{equation}
The following result will be obtained as a consequence of our general abstract Theorem \ref{th1aa}. Precisely, we will show that under an appropriate  choice of  the functions $c(t), \theta(t), \sigma (t), \xi(t)$, the conditions $ (i) - (vii) $ of Theorem \ref{th1aa} are satisfied.
\begin{theorem} \label{ACFR,rescale}{\rm \cite[Theorem 1]{ACFR}}
Let  $x: [t_0, +\infty[ \rightarrow \cH$ be a solution trajectory of \begin{equation*}
{\rm (IGS)_{\alpha/t,\beta,\textit{b}}}\hspace{10mm}
\ddot{x}(t)+\frac{\alpha}{t}\dot{x}(t)+\beta(t)\nabla^2 f(x(t))\dot{x}(t) +b(t)\nabla f(x(t))=0.
 \end{equation*}
  Suppose that  $\alpha > 1$, and that the following growth conditions are satisfied: for $t \geq t_0$
\begin{eqnarray*}
&& (\mathcal{G}_{2}) \quad b(t) > \dot{\beta}(t) +\dfrac{\beta(t)}{t}; 
\\
&& (\mathcal{G}_{3}) \quad  t\dot{w}(t)\leq (\alpha-3)w(t).\hspace{3cm}
\end{eqnarray*}
Then, $w(t)\eqdef b(t)-\dot{\beta}(t) -\dfrac{\beta(t)}{t}$ is positive and 
\begin{eqnarray*}
&& (i) \; \;
f(x(t))-\min_{\cH} f = \cO 
\left(\dfrac{1}{t^2 w(t)}\right) \, \mbox{ as } \, t \to + \infty;
\\
&& (ii) \; \int_{t_{0}}^{+\infty} t \Big(  (\alpha-3)w(t) - t\dot{w}(t)\Big)(f(x(t))-\min_{\cH}f) dt <+\infty; \\
&& (iii) \; \int_{t_{0}}^{+\infty} t^2 \beta(t) w(t)\norm{\nabla f(x(t))}^{2} dt<+\infty.
\end{eqnarray*}
\end{theorem}
{\it Proof } 
Take
$  \, \theta (t)={\rm \Gamma}_{\gamma}(t)^2, \, \, \sigma (t)= \frac{1}{{\rm \Gamma}_{\gamma}(t)},  \, \, \xi (t)\equiv 0,$
and
\begin{equation} \label{def:c}
c(t)^2=  \frac{1}{(\alpha -1)^2}\frac{t}{b(t)}\left( tb(t) -\beta (t) - t \dot{\beta}(t) \right).
\end{equation}
This formula for $c(t)$ will appear naturally during the calculation. Note that the condition $(\mathcal{G}_{2})$ ensures that the second member of the above expression is positive, which makes sense to think of it as a square.
Let us verify that the conditions $(i)$ and $(iv), (v), (vi), (vii)$ are satisfied. This is a direct consequence of the formula 
(\ref{basic_relation_2}) and the condition $(\mathcal{G}_{2})$:

\begin{list}{}{}
\item $(i)$ $\frac{d}{dt}( \beta\theta\sigma)     - \theta b \sigma = \frac{d}{dt} \beta {\rm \Gamma} -{\rm \Gamma}b= \frac{1}{\alpha -1} \left(\frac{d}{dt}(t\beta) -tb   \right)= \frac{t}{\alpha -1} \left(\dot{\beta}  + \frac{\beta}{t}-b   \right) \leq 0$.

\medskip

\item $(iv)$  $\frac{d}{dt}(\theta\sigma) + \theta\sigma \left( \sigma -  \gamma \right) + \xi= \dot{\rm \Gamma} + {\rm \Gamma} \left( \frac{1}{{\rm \Gamma}}-\gamma \right)= \dot{\rm \Gamma} + 1-  \gamma {\rm \Gamma} = 0$.

\medskip

\item $(v)$ Since $\theta \sigma^2 \equiv 1$ and $\xi \equiv 1$, we have  $\frac{d}{dt}(\theta\sigma^2 +\xi) = 0$.

\medskip

\item $(vi)$  $\dot\theta+ 2(\sigma -\gamma)\theta= 2 {\rm \Gamma}
\dot{{\rm \Gamma}} +2 ({\rm \Gamma} -\gamma {\rm \Gamma}^2)= 
2 {\rm \Gamma} (\dot{{\rm \Gamma}} +1- \gamma {\rm \Gamma}) = 0$.

\medskip

\item $(vii)$ $\beta \dot\theta + 2 \left(\dot \beta -b\right) \theta = 2{\rm \Gamma} (\beta \dot{{\rm \Gamma}} + (\dot \beta -b ) {\rm \Gamma})   =  2{\rm \Gamma^2} (  \dot \beta -b  + \frac{\beta}{t}      ) \leq0.   $
\end{list}

\noindent Let's go to the conditions $(ii)$ and $(iii)$. The condition $(iii)$ gives  the formula (\ref{def:c}) for $c(t)$. Then replacing $c(t)^2$ by this value in $(ii)$ gives the condition $(\mathcal{G}_{3})$.  Note then  that $b(t)c(t)^2 = \frac{1}{(\alpha -1)^2}t^2 \omega (t)$, which gives the convergence rate of the values
$$
f(x(t))-\min_{\cH} f = \cO 
\left(\dfrac{1}{t^2 w(t)}\right).
$$
Let us consider the integral estimate for the values. According to the definition (\ref{def:c}) for $c^{2}b$ and the definition of $w$, we have 
\begin{eqnarray*}
\theta b \sigma  - \frac{d}{dt}\left( c^{2}b + \beta\theta\sigma \right)  &=& 
\frac{1}{\alpha -1}tb -\frac{d}{dt}\Big( \frac{1}{\alpha -1}t^2 w(t) + \frac{1}{\alpha -1}t\beta    \Big)\\
&=& 
\frac{t}{(\alpha -1)^2}\Big( (\alpha -1)b -2w -t \dot{w} -(\alpha -1)(\dot{\beta} + \frac{\beta}{t}  )\Big)\\
&=&\frac{t}{(\alpha -1)^2}\Big(  (\alpha-3)w - t\dot{w}\Big).
\end{eqnarray*}
According to Theorem \ref{th1aa} $(ii)$
$$
\int_{t_{0}}^{+\infty} t \Big(  (\alpha-3)w(t) - t\dot{w}(t)\Big)(f(x(t))-\min_{\cH}f) dt <+\infty.
$$
 Moreover, since $\theta \sigma^2 =1$, the formula giving the weighting coefficient $q(t)$ in the integral formula simplifies, and we get
\begin{eqnarray*}
q(t) &=& \theta (t)\sigma (t) \beta (t)\left(  \frac{\dot{\sigma (t)}\beta (t)}{ \sigma^2 (t)} + \frac{b (t)}{ \sigma (t)} -\frac{\dot \beta (t)}{ \sigma (t)}\right)\\
&=& \beta(t) {\rm \Gamma}_{\gamma}(t) \left(-\beta(t) 
\dot{{\rm \Gamma}}_{\gamma} (t) + b(t) {\rm \Gamma}_{\gamma}(t) -        \dot{\beta} (t){\rm \Gamma}_{\gamma}(t) \right) \\
&=& \beta(t) {\rm \Gamma}_{\gamma}(t)^2 \omega(t).
\end{eqnarray*}
According to Theorem  \ref{th1aa} $(iii)$
$$
 \int_{t_{0}}^{+\infty} t^2 \beta(t) w(t)\norm{\nabla f(x(t))}^{2} dt<+\infty
 $$
which gives the announced convergence rates. \qed
\begin{remark}
Take $\beta=0$. Then, according to the definition (\ref{bacic-def}) of $w$, we have 
$w= b$, and the conditions of Theorem \ref{ACFR,rescale} reduce to
$$
t\dot b(t)-(3-\alpha )b(t)\leq 0 \;\; \hbox{ for } t\in [t_0,+\infty[.
$$
We recover the condition introduced in 
\cite[Corollary 3.4]{ACR-Pafa-2019}. Under this condition, each solution  trajectory $x$ of 
\begin{equation*}
{\rm (IGS)_{\alpha/t,0,\textit{b}}}\hspace{10mm}
\ddot{x}(t)+\frac{\alpha}{t}\dot{x}(t) +b(t)\nabla f(x(t))=0,
 \end{equation*}
 satisfies
 $$f(x(t))-\min_{\cH} f = \cO 
\left(\dfrac{1}{t^2 b(t)}\right) \, \mbox{ as } \, t \to + \infty.
$$
\end{remark} 

\subsection{The case $\gamma (t)=\frac{\alpha}{t}$, $\beta$ constant}
Due to its practical importance,  consider the case
$\gamma (t)=\frac{\alpha}{t}$, $\beta (t) \equiv \beta $
where $\beta$ is a fixed positive constant. In this case, the dynamic ${\rm (IGS)}_{\gamma,\beta,b}$ is written as follows
\begin{equation}\label{200}
\ddot{x}(t)+\frac{\alpha}{t}\dot{x}(t)+ \beta\nabla^2 f(x(t))\dot{x}(t)+b(t)\nabla f(x(t))=0.
 \end{equation}
The set of conditions $(\mathcal{G}_{2})$, $(\mathcal{G}_{3})$
boils down to: for $t \geq t_0$ 
\begin{eqnarray*}
&& (\mathcal{G}_{2}) \quad b(t) >  \frac{\beta}{t}; 
\\
&& (\mathcal{G}_{3}) \quad  t\dot{w}(t)\leq (\alpha-3)w(t),\hspace{3cm}
\end{eqnarray*}
where $w(t)= b(t) -\frac{\beta}{t} $.
Therefore, $b(\cdot)$ must satisfy the differential inequality
$$
t \frac{d}{dt} \Big( b(t) -\frac{\beta}{t}   \Big)\leq (\alpha-3) \left( b(t) -\frac{\beta}{t}   \right).
$$
Equivalently
$$
t\frac{d}{dt} b (t) - (\alpha-3)  b(t)+ \beta  (\alpha-2)\frac{1}{t} \leq 0.
$$
Let us integrate  this linear differential equation. Set $b(t)= k(t) t^{\alpha -3}$ where $k(\cdot)$ is an auxiliary function to determine. We obtain 
$$
 \frac{d}{dt} \Big( k(t) -    \frac{\beta}{ t^{\alpha -2}} \Big) \leq 0,
$$
 which gives $k(t)= \frac{\beta}{ t^{\alpha -2}} + d(t)$ with $d(\cdot)$ nonincreasing. Finally,
$
b(t)= \frac{\beta}{t} + d(t) t^{\alpha -3},
$
with $ d (\cdot) $ a nonincreasing function to be chosen arbitrarily.
In summary, we get the following result:
\begin{proposition}\label{Prop_Jordan_coef}
Let  $x: [t_0, +\infty[ \rightarrow \cH$ be a solution trajectory of 
 \begin{equation}\label{2000}
\ddot{x}(t)+\frac{\alpha}{t}\dot{x}(t)+ \beta\nabla^2 f(x(t))\dot{x}(t)+\Big( \frac{\beta}{t} + d(t) t^{\alpha -3} \Big)\nabla f(x(t))=0
 \end{equation}
 where $d(\cdot)$ is a nonincreasing positive function.
Then, the following properties are satisfied:
\begin{eqnarray*}
&& (i) \, \,
f(x(t))-\min_{\cH} f = \cO 
\left(\dfrac{1}{t^{\alpha-1} d(t)}\right) \, \mbox{ as } \, t \to + \infty;
\\
&& (ii) \,  \int_{t_{0}}^{+\infty} -\dot{d}(t) t^{\alpha-1} (f(x(t)) - \inf_{\cH}f ) dt < +\infty.
\\
&& (iii) \, \int_{t_{0}}^{+\infty} t^{\alpha-1} d(t)\norm{\nabla f(x(t))}^{2} dt<+\infty.
\end{eqnarray*}
\end{proposition}
\begin{proof}
According to the definition of $w(t)$ and $b(t)$, we have the equalities\\ 
$t^2w(t)= t^2 \left( b(t) -\frac{\beta}{t}\right)= t^2 d(t) t^{\alpha-3}=  t^{\alpha-1} d(t)$.  Then apply  Theorem \ref{ACFR,rescale}.\qed
\end{proof}  
\subsection{Particular cases}
According to Theorem \ref{ACFR,rescale} and Proposition \ref{Prop_Jordan_coef}, let us discuss the role and the importance of the scaling coefficient $b(t)$ in front of the gradient term.

\medskip

 \textit{a)}
The first inertial dynamic system based on the Nesterov method, and which includes a damping term driven by the Hessian, was considered by Attouch, Peypouquet, and Redont in \cite{APR}. This corresponds to $b(t) \equiv 1$, which gives:

\begin{equation*}
\ddot{x}(t)+\frac{\alpha}{t}\dot{x}(t)+ \beta\nabla^2 f(x(t))\dot{x}(t)+\nabla f(x(t))=0.
 \end{equation*}
 In this case, we have $w(t) =1- \frac{\beta}{t}  $, and we immediately get that  $(\mathcal{G}_{2})$, $(\mathcal{G}_{3})$
 are satisfied by taking $\alpha >3$ and  $t>\beta$.
 This corresponds to take
 $d(t)=\frac{1}{ t^{\alpha -3}}  - \frac{\beta}{ t^{\alpha -2}} $, which is nonincreasing when $t \geq \frac{\alpha -2}{\alpha -3} $.
 
 \begin{corollary}{\rm \cite[Theorem 1.10, Proposition 1.11]{APR}} Suppose that $\alpha > 3$ and $\beta >0$.
Let $x: [t_0, +\infty[ \rightarrow \cH$  be a solution trajectory of 
 \begin{equation}\label{200001}
\ddot{x}(t)+\frac{\alpha}{t}\dot{x}(t)+ \beta\nabla^2 f(x(t))\dot{x}(t)+\nabla f(x(t))=0.
 \end{equation}
Then, 
\begin{eqnarray*}
&& (i) \; \;
f(x(t))-\min_{\cH} f = \cO 
\left(\dfrac{1}{t^{2}}\right) \, \mbox{ as } \, t \to + \infty;
\\
&& (ii)  \; \int_{t_{0}}^{+\infty} t (f(x(t)) - \inf_{\cH}f ) dt<+\infty;
\\
&& (iii) \; \int_{t_{0}}^{+\infty} t^{2} \norm{\nabla f(x(t))}^{2} dt<+\infty.
\end{eqnarray*}
\end{corollary}

\if
{
We can as well take $d(t)=\frac{d}{ t^{\alpha -3}}  - \frac{\beta}{ t^{\alpha -2}} $, where $d$ is a positive parameter, arbitrarily chosen. Therefore, the above results are still valid for the  dynamic
\begin{equation}\label{200001d}
\ddot{x}(t)+\frac{\alpha}{t}\dot{x}(t)+ \beta\nabla^2 f(x(t))\dot{x}(t)+ d \nabla f(x(t))=0.
 \end{equation} 
 }
 \fi
\textit{b)}  Another  important situation is obtained by taking $d(t)=\frac{1}{ t^{\alpha -3}}  $.
This is the limiting case where the following two properties are satisfied:  $d(\cdot)$ is nonincreasing, and the coefficient of $\nabla f(x(t))$ is bounded. This offers  the possibility  of obtaining similar results for the explicit temporal discretized dynamics, that is to say the gradient algorithms.  
Precisely, we obtain the dynamic system considered by
 Shi,  Du,   Jordan, and    Su in \cite{SDJS}, and Attouch, Chbani, Fadili, and Riahi in \cite{ACFR}.

 %
\begin{corollary}{\rm \cite[Theorem 3]{ACFR}, \cite[Theorem 5]{SDJS}}\\  Suppose that $\alpha \geq 3$.
Let $x: [t_0, +\infty[ \rightarrow \cH$ be a solution trajectory of 
 \begin{equation}\label{20000}
\ddot{x}(t)+\frac{\alpha}{t}\dot{x}(t)+ \beta\nabla^2 f(x(t))\dot{x}(t)+\Big( 1 + \frac{\beta}{t}  \Big)\nabla f(x(t))=0
 \end{equation}
Then, the conclusions of Theorem {\rm \ref{ACFR,rescale}} are satisfied:
\begin{eqnarray*}
&& (i) \, \,
f(x(t))-\min_{\cH} f = \cO 
\left(\dfrac{1}{t^{2}}\right) \, \mbox{ as } \, t \to + \infty;
\\
&& (ii) \, \mbox{When } \, \alpha >3, \quad  \int_{t_{0}}^{+\infty} t (f(x(t)) - \inf_{\cH}f ) dt<+\infty.
\\
&& (iii) \, \int_{t_{0}}^{+\infty} t^{2} \norm{\nabla f(x(t))}^{2} dt<+\infty.
\end{eqnarray*}
\end{corollary}

Note that (\ref{20000}) has a slight advantage over (\ref{200001}): the growth conditions  are valid for $t>0$, while for (\ref{200001}) one has to take $t > \beta$. Accordingly, the estimates involve the quantity $\frac{1}{t^{2}}$ instead of $\frac{1}{t^2(1- \frac{\beta}{t}  )}$.

\textit{c)} Take $d(t)= \frac{1}{t^{s}}$ with $s>0$. According to Proposition \ref{Prop_Jordan_coef},
for any  solution trajectory $x: [t_0, +\infty[ \rightarrow \cH$ of 
 \begin{equation}\label{2000bb}
\ddot{x}(t)+\frac{\alpha}{t}\dot{x}(t)+ \beta\nabla^2 f(x(t))\dot{x}(t)+\Big( \frac{\beta}{t} +  t^{\alpha -3-s} \Big)\nabla f(x(t))=0
 \end{equation}
we have:
\begin{eqnarray*}
&& (i) \, \,
f(x(t))-\min_{\cH} f = \cO 
\left(\dfrac{1}{t^{\alpha-1-s}}\right) \, \mbox{ as } \, t \to + \infty;
\\
&& (ii) \,  \int_{t_{0}}^{+\infty}  t^{\alpha-s-2} (f(x(t)) - \inf_{\cH}f ) dt < +\infty, \;
 \int_{t_{0}}^{+\infty} t^{\alpha-s-1} \norm{\nabla f(x(t))}^{2} dt<+\infty.
\end{eqnarray*}

\section{Results based on the function $p_{\gamma}$}
\label{p_gamma}

In this section, we examine another set of growth conditions for the damping and rescaling parameters that guarantee the existence of solutions to  the system  $(i)-(vii)$ of Theorem \ref{th1aa}. In the following theorems, the Lyapunov analysis and the convergence rates are formulated using  the function  $p_{\gamma} : [t_0, +\infty[ \to \R_+$ defined by
$$
p_{\gamma} (t):=\exp\left(\int_{t_0}^{t}\gamma (s)ds\right).
$$
In Theorems \ref{th1ab} and \ref{ACFR,rescale}, in line with the previous articles devoted to these questions (see  \cite{cabot1},  \cite{att2}, \cite{ACR-Pafa-2019}),    the convergence rate of the values was formulated using the function $\Gamma_{\gamma}(t)= p_{\gamma}(t) \int_{t}^{+\infty} \frac{1}{p_{\gamma} (s)}ds$. In fact, each of the two functions $p_{\gamma}$ and $\Gamma_{\gamma}$  captures the properties of the viscous damping coefficient $\gamma(\cdot)$, but   their growths are significantly different.
To illustrate this, in the model case $\gamma (t)=\frac{\alpha}{t}$, $\alpha >1$, 
we have $p_{\gamma}(t)= \left(\frac{t}{t_0}\right)^{\alpha}$, while
$\Gamma_{\gamma}(t) =\frac{t}{\alpha-1}$.
Therefore, $p_{\gamma}$ grows faster than $\Gamma_{\gamma}$ as $t\to +\infty$, and we can expect to get better convergence rates when formulating them using $p_{\gamma}$.
Moreover, $p_{\gamma}$ makes sense and allows to analyze the case $\alpha \leq 1$, while $\Gamma_{\gamma}$ does not.
Thus, we will see that the approach based on $p_{\gamma}$ provides results that cannot be captured by the approach based on $\Gamma_{\gamma}$.
To illustrate this, we start with a simple situation, then we  consider the general case.

\subsection{A model situation}
Consider the system 
 \begin{equation*}
{\rm (IGS)_{\gamma,0,\textit{b}}}\hspace{10mm}
\ddot{x}(t)+\gamma(t)\dot{x}(t) +b(t)\nabla f(x(t))=0
 \end{equation*}
 with 
 $\gamma (t)=\gamma_0(t)+\frac1{{\rm p_0}(t)}\text{ and }\;  {\rm p_0} (t)=\exp\left({ \displaystyle{\int_{t_0}^t \gamma_0 (s)\, ds}}\right).$\\
Choose 
$$\xi \equiv 0,\, \, c(t)=  {\rm p_0} (t), \, \,  \sigma (t)=\frac1{ {\rm p_0} (t)}, \, \, \theta (t)={\rm p_0} (t)^{2}.$$
According to $\dot p_0(t)=\gamma_0(t)p_0(t)$, we can easily verify  that the  conditions $(i),(iii)-(vii)$ of Theorem \ref{th1aa} are satisfied, and $(ii)$ becomes
$
 \frac{d}{dt}\left(  {\rm p_0}(t)^2 b(t)\right) -   {\rm p_0}(t) b (t) \leq 0.
$
Then, a direct application of Theorem \ref{th1aa} gives the following result. 
\begin{theorem}\label{th1ac} 
  Suppose that    for all $t \geq t_0$ 
  \begin{equation}\label{basic_growth_2}
 {\rm p_0}  (t) \dot{b}(t)  +    \Big( 2\gamma_0(t)  {\rm p_0}  (t)-1\Big)b (t)\leq 0.
\end{equation}
Let  $x: [t_0, +\infty[ \rightarrow \mathcal H$ be a solution trajectory of 
${\rm (IGS)_{\gamma,0,\textit{b}}}$. Then,  as $t\to +\infty$
\begin{equation}\label{cont-rap11b}
f (x(t))-\min_\mathcal{H}f= \mathcal O\left(\frac1{ {\rm p_0}(t)^{2} b(t)}\right).
\end{equation}
Moreover, \quad
$
\int_{t_0}^{+\infty}   {\rm p_0}(t) \Big( 1- \left( 2\gamma_0(t)  {\rm p_0}  (t)\right)-  {\rm p_0}  (t) \dot{b}(t)    \Big) (f(x(t))-\min_{\cH} f)dt < +\infty.
$
\end{theorem}

\begin{remark}
Let us rewrite the  linear differential inequality (\ref{basic_growth_2}) as follows:
$$
 \dfrac{\dot{b}(t)}{b(t)}\leq \dfrac{1}{{\rm p_0}(t)}-2\dfrac{\dot{{\rm p_0}}(t)}{{\rm p_0}(t)}.
$$
A solution corresponding to  equality is $ b(t)= {\rm p_0}(t) ^{-2}\exp\Big[\displaystyle\int_{t_{0}}^{t}\Big(\dfrac{1}{{\rm p_0}(s)}\Big) ds\Big]$.\\
In the case   
$\gamma_0(t)(t)=\frac{\alpha}{t}$, $0<\alpha < 1$, $t_0=1$, we have ${\rm p_0}(t)= t^{\alpha}, $ which gives 	
$$	
	b(t)	=  t^{-2\alpha}\exp \Big[\frac{t^{1-\alpha}-1}{1-\alpha}\Big].
 $$
	Therefore, for $ 0 < \alpha< 1,$  and for this choice of $b$, (\ref{cont-rap11b}) gives
	\begin{equation}
	f(x(t))-\min_\mathcal{H}f=\mathcal{O}\Big(\frac{1}{\exp \Big[\frac{t^{1-\alpha}}{1-\alpha}\Big] }\Big).
	\end{equation}
	Thus, we obtain an exponential convergence rate in a situation that cannot be covered by the $\Gamma_{\gamma}$  approach.
\end{remark}

\subsection{The general case, with the Hessian-driven damping}

\begin{theorem}\label{th1}
		Let $f: \mathcal{H} \rightarrow \R $  be a convex function of class $\cC^1$  such that $\argmin_{\cH}f\neq\emptyset$. Suppose  that $\gamma(\cdot), \beta(\cdot)$ are  $\cC^1$ functions and $b(\cdot)$ is a $\cC^2$ function which is nondecreasing. Suppose that $r$ and $m$ are positive  parameters which satisfy 
		$0<r\leq \frac{1}{3}$ and $ 2r\leq m\leq 1-r$. Suppose that the following growth conditions are satisfied: for $t\geq t_0$ 
		\begin{eqnarray*}\label{condition}
			(\mathcal{H}_{1}) &&\xi_0(t)\geq 0;\\
			(\mathcal{H}_{2}) &&  
			\xi_0(t) \big(2\sigma(t)-(m+r)\gamma(t)\big) -\frac{1}{2}\dot{\xi}_0(t)+\big(m-(1-r)\big)\gamma(t)\sigma^{2}(t)\geq 0,\\
			(\mathcal{H}_{3})&& b(t)-\dot{\beta}(t)+\beta(t)\big(\sigma(t)-\gamma(t)\big)\geq0,\\
			(\mathcal{H}_{4}) && 
			\frac{d}{dt}\Big(\theta(w+\beta\sigma) \Big)(t)-\theta(t)b(t)\sigma (t)\leq 0.
		\end{eqnarray*}
		where 
		\begin{eqnarray}
		&& \xi_0(t):=\big((1-2(r+m))\gamma(t)+\sigma(t)\big)\sigma(t)-\dot\sigma(t), \label{def:xi0}\\
		&& \sigma(t):=m\gamma(t)+\frac{1}{3}\frac{\dot{b}(t)}{b(t)}. \label{def:sigma}\\
		&& w(t)=b(t)-\dot{\beta}(t)+\beta(t)\sigma(t)+(1-2r-2m)\gamma(t)\beta(t).
		\end{eqnarray}
		Then, for  each  solution trajectory of $x: [t_0, +\infty[ \rightarrow \cH$ of  ${\rm (IGS)_{\gamma,\beta,\textit{b}}}$, we have,
		\begin{eqnarray}
	&(i)&  \, 	f(x(t))-\min_\mathcal{H}f= 
		\mathcal O\left(\frac1{p_{\gamma}(t)^{2r}w(t)b(t)^{\frac{-2}{3}}}\right) \mbox{ as } t \to +\infty		\label{cont-rap111}\\
	&(ii)&	 \,    \int_{t_0}^{+\infty} p_{\gamma}^{2r}(t) \Upsilon(t)\left(f(x(t))- \inf_{\cH} f \right) dt < +\infty ,	\label{cont-rap111b}\\
	&(iii)&	\,\int_{t_0}^{+\infty}\! \left(p_{\gamma}^{2r}(t) b^{\frac13}(t)\beta (t) -  \frac{d}{dt}\big(p_{\gamma}^{2r} b^{-\frac23} \beta ^2 \big)(t)\!\right)\|\nabla f(x(t))\|^2 dt < +\infty. 	\label{cont-rap111c}
		\end{eqnarray}
Here $\Upsilon (t) \eqdef \Big(3\sigma(t)-2(r+m)\gamma(t)\Big)w(t)-\dot{w}(t)-2(1-r-m)\gamma(t)$.
	\end{theorem}
	{\it Proof }
	According to Theorem \ref{th1aa}, it suffices to show that, under the hypothesis $(\mathcal{H}_{1})-(\mathcal{H}_{4}),$  there exists $c, \theta, \sigma, \xi$  which satisfy the conditions $(i)-(vii)$ of Theorem \ref{th1aa}. 
	To perform the corresponding derivative calculation, let's  start by establishing some preliminary results.
	
	\smallskip 
	
	$\bullet$ $\ln p_{\gamma}(t) = \int_{t_0}^{t}\gamma (s)ds$, which by derivation gives $\frac{\dot{p_{\gamma}}}{p_{\gamma}}= \gamma$, that is to say
	$\dot{p_{\gamma}} =  \gamma p_{\gamma}.$

	\smallskip 
		
	$\bullet$ According to  the definition of $\sigma$,
	\begin{eqnarray}
	\frac{d}{dt} \left(p_{\gamma} ^{2r}b^{-\frac{2}{3}}\right) &=&2p_{\gamma}^{2r}b^{-\frac{2}{3}}\left(r\gamma-\frac{1}{3}\frac{\dot{b}}{b}\right) \label{derivation_c0} \\
	&=&2\theta\left((r+m)\gamma-\sigma\right). \label{derivation_c}
	\end{eqnarray}
	Let us show that  the following choice of the unknown parameters $c, \theta, \sigma, \xi$  satisfies the conditions $(i)-(vii)$ of Theorem \ref{th1aa}:
	$$ \theta:=p_{\gamma}^{2r}b^{-\frac{2}{3}}, \quad \sigma:=m\gamma+\frac{1}{3}\frac{\dot{b}}{b}, \quad \xi:=\theta\xi_0 ,$$  and  	
	\begin{equation}\label{cb}
	c^{2}b:= \theta w:= \theta\Big(b-\dot{\beta}+\beta\sigma+(1-2r-2m)\gamma\beta\Big),
	\end{equation}
	where $\xi_0$ has been defined in (\ref{def:xi0}). 
We  underline that under condition $(\mathcal{H}_{3}),$ 
 $$c^{2}b=\theta\Big(\underbrace{b-\dot{\beta}+\beta\sigma-\gamma \beta}_{\geq 0}+2\underbrace{(1-r-m)}_{\geq 0}\gamma\beta\Big)\geq 0.$$
Also, according to (\ref{derivation_c}), we have 
%
$\dot{\theta}=2\theta\big((r+m)\gamma-\sigma\big).$
	\begin{list}{}{}		
	\item \begin{equation}\label{ii}
		\begin{array}{lll}
(i)	&&\frac{d}{dt}( \beta\theta\sigma)     - \theta b \sigma		=   \dot{\beta}\theta\sigma+\beta(\dot{\theta}\sigma+\theta\dot{\sigma})-b\theta\sigma \\ 
			&  &= \theta\Big[\dot{\beta}\sigma+2\Big((r+m)\gamma-\sigma\Big)\beta\sigma+\beta\dot{\sigma}-b\sigma\Big]\quad  because \quad \dot{\theta}=2((r+m)\gamma-\sigma)\\ 
			&  &= \theta \Big( -\beta\xi_0-\sigma(b-\dot{\beta}+\beta\sigma-\gamma\beta)\Big) .
			\end{array} 
	\end{equation}
Since $b$ is nondecreasing, then $\sigma\geq 0,$
		so by  $(\mathcal{H}_{1})$ and  $(\mathcal{H}_{3}),$ we get  $$\frac{d}{dt}( \beta\theta\sigma)  - \theta b \sigma =\theta \Big( -\beta\xi_0-\sigma\underbrace{(b-\dot{\beta}+\beta\sigma-\gamma\beta)}_{\geq0}\Big)\leq 0$$

		\smallskip
		
		\item $(ii)$ According to 
		the derivation chain rule
 and $(\mathcal{H}_{4})$, we conclude that 
		\begin{eqnarray*}
		\frac{d}{dt}\left( c^{2}b  \right) &+ & \frac{d}{dt}(\beta\theta\sigma)    - \theta b \sigma = \frac{d}{dt}\left( \theta w\right)   +\frac{d}{dt}(\beta\theta\sigma) - \theta b \sigma\\
			&=& \frac{d}{dt}\left( \theta w  + \beta\theta\sigma\right)  - \theta b \sigma \leq 0.
		\end{eqnarray*}
	 		
		\smallskip
		
		\item $(iii)$ $ b(c^{2}-\theta )  + \beta \theta  (\sigma-\gamma) + \frac{d}{dt}( \beta\theta) = 0$ \; results from (\ref{cb}).
		
		\smallskip
		
		\item $(iv)$   According to the derivation chain rule, (\ref{derivation_c0}), and the definition of $\sigma$
		\begin{eqnarray*}
			\frac{d}{dt}(\theta \sigma) + \theta \sigma \left( \sigma -  \gamma \right) + \xi &=& \dot{\theta} \sigma + \theta \dot{\sigma}+ \theta \sigma
			\left( \sigma -  \gamma \right) +   \xi\\
			&=& 2\theta\sigma\big((r+m)\gamma-\sigma\big)+\theta \dot{\sigma}+ \theta \sigma
			\left( \sigma -  \gamma \right) +   \xi \\
			&=& \theta\Big(\dot{\sigma}-\sigma\big((1-2r-2m)\gamma+\sigma\big) \Big) + \xi.
		\end{eqnarray*}
		For this quantity to be equal to zero, we therefore  take  $\xi = \theta \xi_0$, 
	 	where $\xi_0$ is defined in (\ref{def:sigma}).
		
		\smallskip
		
		\item $(v)$ According to   our choice  
		$\xi=\theta\xi_0$,   
		we have
		$$
		(v) \Longleftrightarrow   \frac{d}{dt}\left(\theta(\sigma^2+\xi_0)\right) \leq 0.
		$$
		Let's compute this quantity. 
		According to the derivation chain rule and (\ref{derivation_c}) 
		\begin{eqnarray*}
			\frac{d}{dt}\left(\theta(\sigma^2+\xi_0)\right) &=& \Big(	\theta\dot\xi_0 + \dot \theta (\sigma^2 + \xi_0)+ 2\theta\dot\sigma\sigma \Big) \\
			& = &  	2\theta \Big(	\demi \dot\xi_0 +  (\sigma^2 + \xi_0)\left((r+m)\gamma-\sigma\right)+ \dot\sigma\sigma \Big) \\
			& = &  2\theta \Big(	\demi \dot\xi_0 + \xi_0\big((r+m)\gamma-2\sigma\big)+ \sigma \left(\xi_0 + \sigma\big((r+m)\gamma-\sigma  \big) +\dot\sigma \right) \Big).
		\end{eqnarray*}
		By definition of  $\xi_0$, we have $\xi_0 + \dot\sigma =\Big((1-2(r+m))\gamma+\sigma\Big)\sigma$. Therefore 
		\begin{eqnarray*}
			\frac{d}{dt}\left(\theta(\sigma^2+\xi_0)\right)
			& = &  2\theta \Big(	\demi \dot\xi_0 + \xi_0\left((r+m)\gamma-2\sigma\right)+ \gamma\sigma^2 \left( 1-(r+m)\right)   \Big).
		\end{eqnarray*}
		So, $(v)$ is satisfied under the condition 
		$$
		\demi \dot\xi_0 + \xi_0\big((r+m)\gamma-2\sigma\big)+ \gamma\sigma^2 \big( 1-(r+m)\big)  \leq 0,
		$$
		which is precisely  $(\mathcal{H}_{2})$.
		
		\smallskip
		
		\item $(vi)$ Let's compute
		\begin{eqnarray*}
			\dot\theta+ 2(\sigma -\gamma)\theta
			& = &   2\theta\Big(	r\gamma-\frac{1}{3}\frac{\dot{b}}{b}
			+(m-1)\gamma  +\frac{1}{3}\frac{\dot{b}}{b} \Big)\\
			& = &  2\left(r+m-1\right)\theta \gamma . 
		\end{eqnarray*}
		According to the assumption $m\leq 1-r$, this quantity is less or equal than zero.
		
		\smallskip
		
		\item We have 
			$(vii)	 \Longleftrightarrow \beta (\beta\dot{\theta}+2(\dot{\beta}-b)\theta) \leq0 .$
			
	\smallskip
				
		According to condition $\mathcal{H}_{3}$ and the assumption $m\leq 1-r,$  we conclude  
		\begin{eqnarray*}
		 \beta\dot{\theta}+2(\dot{\beta}-b)\theta &=& 2\theta\big(\beta(r+m)\gamma-\beta\sigma-b\big)\\
		 &=&2\theta\Big[-\underbrace{(b-\dot{\beta}+\beta\sigma-\gamma\sigma)}_{\geq 0}-\beta\gamma\underbrace{(1-r-m)}_{\geq 0}\Big]\leq 0.
		 \end{eqnarray*}
	 So, $(vii)$ is satisfied
	\end{list}
According to Theorem \ref{th1aa}, we obtain 
(\ref{cont-rap111})-(\ref{cont-rap111b})-(\ref{cont-rap111c})
which completes the proof.   \qed

 \subsection{The case without the Hessian}
  Let us specialize the previous results in  the case $\beta =0$, \ie without the Hessian:
  
  \smallskip
  
\begin{center}
$
{\rm (IGS)_{\gamma,0,\textit{b}}}\hspace{10mm}
\ddot{x}(t)+\gamma(t)\dot{x}(t) +b(t)\nabla f(x(t))=0.
$
\end{center}

 \begin{theorem}\label{th1_H}
 Suppose that the conditions $(\mathcal{H}_{1})$  and $(\mathcal{H}_{2})$ of Theorem {\rm \ref{th1}} are satisfied.
Then, for  each  solution trajectory $x: [t_0, +\infty[ \rightarrow \cH$ of  ${\rm (IGS)_{\gamma,0,\textit{b}}}$, we have, as $t \to +\infty$
\begin{equation}\label{cont-rap111_H}
f(x(t))-\min_\mathcal{H}f= 
\mathcal O\left(\frac1{p_{\gamma}(t)^{2r}b(t)^{\frac{1}{3}}}\right).
\end{equation}
Moreover, when $m>2r$
\begin{equation}\label{cont-rap111b_H}
\int_{t_0}^{+\infty} p_{\gamma}(t)^{2r}b(t)^{\frac{1}{3}} \gamma (t) \left(f(x(t))- \inf_{\cH} f \right) dt < +\infty.
\end{equation}
\end{theorem}
	{\it Proof }
Conditions  $ (\mathcal{H}_{1})$  and  $ (\mathcal{H}_{2})$  in Theorem \ref{th1}  remain unchanged since they are independent of $ \beta $. We just need to verify  $ (\mathcal{H}_{4})$, because $ (\mathcal{H}_{3})$    is written $b(t)\geq 0$ and becomes obvious.
Since $\beta=0$, we have 
$
(\mathcal{H}_{4}) \Longleftrightarrow 
	\frac{d}{dt}\Big(\theta b \Big)(t)-\theta(t)b(t)\sigma (t)\leq 0.
$
According to 
\begin{eqnarray*}
	\frac{d}{dt}\Big(\theta b \Big)(t)-\theta(t)b(t)\sigma (t)
	&= &
	\dot \theta(t)b(t)+\theta(t)\dot  b(t)-\theta(t)b(t)\left( m\gamma+\frac{\dot b(t)}{3 b(t)}\right)\\
	&=&
	b(t)^{1/3}\left[\frac{d}{dt}\Big(\theta(t)b(t)^{2/3}\Big) -  m\gamma(t)\Big(\theta(t)b(t)^{2/3}\Big)\right]\\
	&= &
	b(t)^{1/3}\left[\frac{d}{dt}\Big(p_\gamma(t)^{2r}\Big) -  m\gamma(t)\Big( p_\gamma(t)^{2r}\Big)\right]\\
	&=&
	(2r -  m)\gamma(t) b(t)^{1/3} p_\gamma(t)^{2r}\leq 0 \;\;\text{ since }2r \leq  m,
\end{eqnarray*}
we conclude that $ (\mathcal{H}_{4})$ holds, which completes the proof.   \qed

Next, we show that   the condition $ (\mathcal{H}_{2})$  on the coefficients $\gamma(\cdot)$ and $b(\cdot)$ can be formulated in  simpler form which is useful in practice.
 \begin{theorem}\label{th1a}
The conclusions  of Theorem {\rm \ref{th1_H}}  remain true when we replace  $(\mathcal{H}_{2})$ by 
\begin{description}\label{condition+}
\smallskip
\item[$(\mathcal{H}_{2}^+)$]  \quad
$\sigma(t)\big( \sigma(t)-(r+m)\gamma(t)\big) \big( 2\sigma(t)+\big(1-2(r+m)\big)\gamma(t) \big) +\frac{1}{2}\ddot\sigma(t) \geq 0 $,
\end{description}
and assume moreover that 
 $b(\cdot)$ is   log-concave, i.e., $\frac{d^2}{dt^2}(\ln(b(t)))\leq 0$. 
\end{theorem}
{\it Proof }
 According to Theorem \ref{th1_H}, it suffices to show that  $(\mathcal{H}_{2})$  is satisfied under the hypothesis $(\mathcal{H}_{2}^+)$. By definition of $\sigma$,
we have  
$$\big(2\sigma(t)-(m+r)\gamma(t)\big)= \Big( (m-r)\gamma(t) + \frac{2}{3}\frac{\dot{b}(t)}{b(t)}\Big)  .$$
  So $(\mathcal{H}_{2})$ can be written equivalently as $\mathcal A \geq 0$, where
 \begin{equation}\label{condition2b}
\mathcal A =:
\xi_0(t) \Big( (m-r)\gamma(t) + \frac{2}{3}\frac{\dot{b}(t)}{b(t)}\Big) -\frac{1}{2}\dot{\xi}_0(t)+\Big(m+r-1\Big)\gamma(t)\sigma^{2}(t).
\end{equation}
A calculation similar to the one above gives  
 \begin{eqnarray}
 \xi_0(t)&=&\Big((1-2r-m)\gamma(t)- m\gamma(t) +\sigma(t)\Big)\sigma(t)-\dot\sigma(t), \nonumber  \\
&= & \Big((1-2r-m)\gamma(t)+ \frac{1}{3}\frac{\dot{b}(t)}{b(t)} \Big)\sigma(t)-\dot\sigma(t).\label{def:xi0_b}
\end{eqnarray}
In (\ref{condition2b}), let's replace $\xi_0(\cdot)$ by its formulation (\ref{def:xi0_b}), we obtain 
 \begin{eqnarray*}
 \mathcal A  
 & = & \frac{1}{2}\frac{d^2}{dt^2}\sigma(t) 
  -\frac{1}{2}\frac{d}{dt}\left[\sigma(t)\left((1-2r-m)\gamma(t)+\frac{1}{3}\frac{\dot{b}(t)}{b(t)}\right)\right]\\
& -& \dot\sigma(t)\left((m-r)\gamma(t)+\frac{2}{3}\frac{\dot{b}(t)}{b(t)}\right) + \Big(m+r-1\Big)\gamma(t)\sigma^{2}(t)
\\
&+&\Big( (m-r)\gamma(t) + \frac{2}{3}\frac{\dot{b}(t)}{b(t)}\Big)\Big((1-2r-m)\gamma(t)+ \frac{1}{3}\frac{\dot{b}(t)}{b(t)} \Big)\sigma(t).
  \end{eqnarray*}
 Set
 \begin{eqnarray*}
 \mathcal B &:=&\Big(m+r-1\Big)\gamma(t)\sigma^{2}(t)
+\Big( (m-r)\gamma(t) + \frac{2}{3}\frac{\dot{b}(t)}{b(t)}\Big)\Big((1-2r-m)\gamma(t)+ \frac{1}{3}\frac{\dot{b}(t)}{b(t)} \Big)\sigma(t),\\
&& 
  \end{eqnarray*}
then we have (by omitting the variable $t$ to shorten the formulas)
 \begin{eqnarray*}
 \mathcal B &=& \sigma    \Big[  (m+r-1)\gamma \sigma
+\left( (m-r)\gamma(t) + \frac{2}{3}\frac{\dot{b}}{b}\right)\left((1-2r-m)\gamma+ \frac{1}{3}\frac{\dot{b}}{b} \right)\Big]\\
 &=& \sigma    \Big[  (m+r-1)\gamma \sigma
+\left( - r\gamma + \frac{1}{3}\frac{\dot{b}}{b} + \sigma\right)\left((1-2r)\gamma+ \frac{2}{3}\frac{\dot{b}}{b} -\sigma \right)\Big]\\
 &=& \sigma    \Big[  (m+r-1)\gamma \sigma - \sigma^2 + \gamma\sigma \left( -m+1-r  \right)+\sigma^2
+\left( - r\gamma + \frac{1}{3}\frac{\dot{b}}{b} \right)\left((1-2r)\gamma+ \frac{2}{3}\frac{\dot{b}}{b}\right)\Big]\\
&=& \sigma    
\left(- r\gamma + \frac{1}{3}\frac{\dot{b}}{b} \right)\left((1-2r)\gamma+ \frac{2}{3}\frac{\dot{b}}{b}\right).
  \end{eqnarray*}
Replacing $ \mathcal B$ in  $\mathcal A$, we obtain
 \begin{equation}\label{def:A}
 \mathcal A  
  =  \sigma(t)\Big(\sigma(t) -(m+r)\gamma(t)\Big) \Big(
 2 \sigma(t) +(1-2(m+r))\gamma(t)\Big) +\frac{1}{2}\frac{d^2}{dt^2}\sigma(t) + C(t)
  \end{equation}
where
 \begin{eqnarray*}
C (t)
  & :=& -\dot\sigma(t)\left((m-r)\gamma(t)+\frac{2}{3}\frac{\dot{b}(t)}{b(t)}\right) -\frac{1}{2}\frac{d}{dt}\left[\sigma(t)\left((1-2r-m)\gamma(t)+\frac{1}{3}\frac{\dot{b}(t)}{b(t)}\right)\right] .
\end{eqnarray*} 
Let us  show that $C(t)$ is nonnegative. After replacing $\sigma (t)$ by its value $m\gamma(t)+\frac{1}{3}\frac{\dot{b}(t)}{b(t)}$, and developing, we get 
 \begin{eqnarray*}
C (t)
  &=& -m\dot{\gamma}(t)\gamma(t)(1-3r) - \frac{1}{6} (4m-2r+1)\dot{\gamma}(t)\frac{\dot{b}(t)}{b(t)} \\
  && -\frac{1}{6}  \frac{d^2}{dt^2}\left(\ln(b(t))\right)\left((1+2(m-2r))\gamma(t)+ 2\frac{\dot{b}(t)}{b(t)}\right) .
\end{eqnarray*} 
By assumption,  $m-2r\geq 0$, $1-3r \geq 0$, $\gamma(\cdot)$ is nonincreasing,  $b(\cdot)$ is nondecreasing,  and  $\frac{d^2}{dt^2}\left(\ln(b(t))\right)\leq 0$. We  conclude that
$C(t) \geq 0$. According to (\ref{def:A}), we obtain
$$
\mathcal A  
\geq \sigma(t)\Big(\sigma(t) -(m+r)\gamma(t)\Big) \Big(
 2 \sigma(t) +(1-2(m+r))\gamma(t)\Big) +\frac{1}{2}\frac{d^2}{dt^2}\sigma(t). 
$$
The condition  
$ (\mathcal{H}_{2}^+)$   expresses that the second member of the above inequality is nonnegative. Therefore  $ (\mathcal{H}_{2}^+)$  
implies $ (\mathcal{H}_{2})$, which  gives the claim.
 \qed
\subsection{Comparing the two approaches}
 As we have already underlined, Theorems \ref{th1ab} and \ref{th1a} are based on the Lyapunov analysis of the dynamic ${\rm (IGS)}_{\gamma,0,b} $ using the functions  $\Gamma_{\gamma}$ and $p_{\gamma}$, respectively. As such, they lead to significantly different growth conditions on the coefficients of the dynamic. Precisely, using the following example, we will show that  Theorem \ref{th1a} better captures the case where $ b $ has an exponential growth. Take
$$
b(t)=e^{\mu t^{q}}\hbox{ and }\gamma(t)=\frac{\alpha}{t^{1-q}}\hbox{ with }\alpha=\mu q>0,\; q\in (0,1).
$$ 
\textit{a)} First, let us show that the condition $(\mathcal{H}_{2}^+)$ of Theorem \ref{th1a}  is satisfied. We have
\begin{eqnarray*}
&&\frac{1}{2}\ddot\sigma(t) + \sigma(t)\Big(\sigma(t) -(m+r)\gamma(t)\Big) \Big(
 2 \sigma(t) +(1-2(m+r))\gamma(t)\Big)\\
 &&=  (\mu q)^3\left(m+\frac13\right) \left(\frac13-r\right) \left(\frac53-2r\right)\frac{1}{ t^{3-3q}} + \demi \mu q\left(m+\frac13\right)(1-q)(2-q) \frac{1}{ t^{3-q}} 
\end{eqnarray*}
which is nonnegative because of the hypothesis $r \leq \frac13$ and $q<1$.

\smallskip
 
\noindent \textit{b)}  Let us now examine  the  growth condition  used in Theorem  \ref{th1ab}:
 \begin{equation}\label{basic-cond2-bis-b3}
\Gamma (t) \dot{b}(t)  \leq   b(t) \Big( 3-2\gamma(t) \Gamma (t)\Big) \hbox{ where } \Gamma (t):=p(t)\int_{t}^{+\infty}\frac{ds}{p(s)}.
\end{equation}
 Here $pt)=e^{\mu(t^q-t_0^q)}$. Therefore $\Gamma (t)= e^{\mu t^q}\displaystyle\int_{t}^{+\infty}e^{-\mu s^q}ds$, which gives
\begin{eqnarray*}
\Gamma (t) \dot{b}(t) - b(t) \Big( 3-2\gamma(t) \Gamma (t)\Big) &= &
  3e^{\mu t^q}\left( \mu qt^{q-1}e^{\mu t^q}\displaystyle\int_{t}^{+\infty}e^{-\mu s^q}ds - 1\right).
\end{eqnarray*}
Let us analyze the sign of the above quantity, which is the same as
\begin{eqnarray*}
\mathcal D (t) &:=&  \mu q t^{q-1}e^{\mu t^q}\displaystyle\int_{t}^{+\infty}e^{-\mu s^q}ds - 1\\
&=&  -\mu q t^{q-1}e^{\mu t^q}\displaystyle\int_{t}^{+\infty}
\frac{d}{ds}\left( e^{-\mu s^q}\right) \frac{1}{ \mu q} s^{1-q} ds - 1
\end{eqnarray*}
After integration by parts, we get
\begin{eqnarray*}
\mathcal D (t) &:=&  \left( \frac{1}{q} -1\right) +  \frac{1-q}{q}t^{q-1}e^{\mu t^q}\displaystyle\int_{t}^{+\infty} e^{-\mu s^q} \frac{1}{s^q}ds > \left( \frac{1}{q} -1\right) >0.
\end{eqnarray*}
Therefore, the condition (\ref{basic-cond2-bis-b3})  is not satisfied.

 \section{Illustration of the results}\label{illustrations}
Let us  particularize our results in some important special cases, and compare them with the existing litterature. We do not detail the proofs which result from the direct applications of the previous theorems and the classical differential calculus.

 \subsection{The case $b(t)= p (t)^{3p_0}$.}
Recall that $p (t)=\exp\left(\displaystyle\int_{t_{0}}^{t}\gamma(s)ds\right)$.  We start with results in  \cite{att2} concerning the rate of convergence of values in the case 
$b(t)= c_0 p(t)^{3p_0}$ with $p_0\geq 0$ and $c_0\geq 0$. In this case, the  system ${\rm (IGS)_{\gamma,0,\textit{b}}}$ becomes:
   \begin{equation}\label{equaexp}
  \ddot{x}(t)+\gamma(t)\dot{x}(t)+c_0\exp\left(3p_0\displaystyle\int_{t_{0}}^{t}\gamma(s)ds\right)\nabla f(x(t)=0. 
  \end{equation}
Observe that $\frac{\dot{b}(t)}{3b(t)}= p_0\gamma(t)$ and  
  $\xi_0(t)=(m+p_0) \left((1-2r-m+p_0)\gamma^2(t)  - \dot\gamma(t)\right)$. 
Therefore, conditions    ($\mathcal{H}_{1})$ and ($\mathcal{H}_{2})$  of Theorem \ref{th1_H} become after simplification:  
 \begin{description}\label{condition_c}
\item[($\mathcal{H}_{1}$)]
$[(p_0-r)+ (1-r-m)] \gamma^2(t)  - \dot\gamma(t)\geq 0$;

\smallskip

\item[($\mathcal{H}_{2}$)]
$2(p_0-r)\left(1+2(p_0-r)\right)\gamma^3(t)  - 2\left(1+3(p_0-r)\right)\gamma(t)\dot\gamma(t)+ \ddot\gamma(t) \geq 0 $.
\end{description}
Since $m\leq 1-r$, instead of   $(\mathcal{H}_{1})$, it suffices to verify 
 \begin{description}\label{condition_b}
  \item[($\mathcal{H}_{1}^{+}$)] $\big(p_0-r\big)\gamma^{2}(t)-\dot{\gamma}(t)\geq0$.
\end{description}
  \begin{theorem}\label{main}
   Let $\gamma: [t_0 , +\infty)\rightarrow \R_{+} $ be a nonincreasing  and twice continuously differentiable function. Suppose that  there exists  $r\in (0, \frac{1}{3}\big]$  such that 
\begin{equation}\label{eq61}
\ddot{\gamma}(t) \geq 2\big[\min (0,p_0-r)\big]^{2}\gamma^{3}(t) \hbox{ on } \, [t_0 , +\infty) .
\end{equation}
 Then, for each solution trajectory  $x(\cdot)$ of  {\rm(\ref{equaexp})}, we have   as $t \to +\infty$ 
\begin{equation}\label{eq62}
f(x(t))-\min_\mathcal{H}f= \mathcal O\left(\frac1{p (t)^{2r+p_0}}\right).
\end{equation}

\if
{
{\bf (b)} Suppose moreover that  $r\neq\frac{1}{3}$, then  
    \begin{itemize}
    \item[\bf (i)] When $\displaystyle\int_{t_{0}}^{+\infty}\gamma(t)dt=+\infty$, we have  $f(x(t))-\min_\mathcal{H}f= o\left(\frac1{p (t)^{2r+p_0}}\right)$;
  \item[\bf (ii)]  When \, $p(t)^{1+2(p_0-r)} 
  {\displaystyle \int_{t}^{+\infty}\frac{dt}{p(t)} } \leq \gamma(t) $, then   $x(\cdot)$   converges weakly  toward some $x^{*} \in \argmin_{\mathcal{H}}f$.
    \end{itemize} 
    }
    \fi
  \end{theorem}
 {\it Proof } 
  To prove the claim, we use Theorem \ref{th1} and distinguish two cases: 

 $ \star$ Suppose $p-r\geq 0$, then (\ref{eq61}) implies $\ddot{\gamma}(t)\geq0$, and  since $\gamma$ is a nonincreasing, we also have  $\dot{\gamma}(t) \leq 0$;  thus both conditions $(\mathcal{H}_{1}^+)$ and $(\mathcal{H}_{2})$ are  satisfied.

\smallskip

 $ \star$ Suppose  $p-r< 0$, then (\ref{eq61}) becomes
 \begin{equation}\label{eq-42}
 \ddot{\gamma}(t) \geq (2p-r)^{2}\gamma^{3}(t) \hbox{ on }[t_0 , +\infty).
 \end{equation} 
  Since  $\gamma(\cdot)$ is a positive and nonincreasing, $\lim_{t\rightarrow +\infty}\gamma(t)=\ell$ exists and is equal to zero. Otherwise,  by integrating  (\ref{eq-42}) on $[t_0,t]$ for $t>t_0$, we would have   
$$
\dot{\gamma}(t)-\dot{\gamma}(t_0) \geq 2(p-r)^{2}\int_{t_0}^t\gamma(s)^{3}ds \geq 2(p-r)^{2}\ell^{3}(t-t_0).
$$ 
This in turn gives   $\lim_{t\rightarrow +\infty}\dot{\gamma}(t)=+\infty$, which implies $\lim_{t\rightarrow +\infty}\gamma(t)=+\infty$, that is a contradiction.  
Then, multiply  (\ref{eq-42})   by $\dot{\gamma}(t)$. Since $\gamma(\cdot)$ is  nonincreasing, we obtain
  $$
  \ddot{\gamma}(t)\dot{\gamma}(t)\leq 2(p-r)^{2}\gamma^{3}(t)\dot{\gamma}(t) \Longleftrightarrow \frac{1}{2} \frac{d}{dt}(\dot{\gamma}(t)^{2}) \leq\frac{(p-r)^{2}}{2} \frac{d}{dt}(\gamma^{4}(t)).
  $$
  By integrating this inequality from $t$ to $T > t$, we get
$$
\dot{\gamma}(T)^{2}-\dot{\gamma}(t)^{2}\leq (p-r)^{2} (\gamma^{4}(T)-\gamma^{4}(t)),
$$
  Letting $T \rightarrow +\infty,$ and using $\lim_{T \rightarrow +\infty}\gamma(T)=0$, we obtain $\dot{\gamma}^{2}(t)\geq (p-r)^{2}\gamma^{4}(t),$  which is equivalent to $|\dot{\gamma}(t)|\geq |p-l|\gamma^{2}(t).$ Since $\dot{\gamma}(t)\leq 0$ and $p< r$, this gives  
  $-\dot{\gamma}(t)\geq (r-p)\gamma^{2}(t),\; \forall t>t_0$,
that is ($\mathcal{H}_{1}^{+}).$ We have
\begin{eqnarray*}
&&[(p-r)+ (1-r-m)] \gamma^2(t)  - \dot\gamma(t)\\
&&=  
 \underbrace{-2(p-r)^{2}\gamma^{3}(t)+\ddot{\gamma}(t)}_{\geq 0 \text{ by } (\ref{eq-42})}+2 \underbrace{(1-3r+3p)}_{\geq 0 \text{ since } p<r}\gamma(t)\big(\underbrace{(p-r)\gamma^{2}(t)-\dot{\gamma}(t)}_{\geq 0 \text{ by } (\mathcal{H}_{1}^{+})}\big)
  \geq0. 
 \end{eqnarray*}
  Therefore,  ($\mathcal{H}_{1}^{+})$ and ($\mathcal{H}_{2})$ are satisfied. Applying  Theorem \ref{th1}, we conclude. \qed
 
\smallskip

 As a particular case  of Theorem \ref{main}, with $p_0=0$, we obtain the following result. 
 
 \begin{theorem} \label{theorem-basic} {\rm \cite[Theorem 2.1]{att2}}
Let  $\gamma(\cdot)$ be a nonncreasing  function of class $\cC^2$, and $x(\cdot)$ a solution trajectory 
    of 
     \begin{equation}\label{equaexp_0}
  \ddot{x}(t)+\gamma(t)\dot{x}(t)+c_0\nabla f(x(t)=0. 
  \end{equation} 
Suppose  that
  \begin{description}\label{eq5}
\item[$(\mathcal{H}_{r,\gamma})$] \quad
$\exists r>0$ such that  $-2r^{2}\gamma^{3}(t)+\ddot{\gamma}(t)\geq 0$  for $t$ large enough.
\end{description}
 Then, \quad 
$f(x(t))-\min_\mathcal{H}f=  \mathcal O\left(e^{-2\min (r,\frac13)\int_{t_{0}}^{t}\gamma(s)ds}\right)$ as $t \to +\infty$.
\end{theorem}

\begin{remark}
The case $\gamma(t)=\frac{1}{t (\ln t)^{\rho}}$, for $0 \leq \rho \leq 1$,  was developed in  \cite{att2}. 
In that case
condition $(\mathcal{H}_{3,\gamma})$ writes as 
$$
2(\ln t )^2 + 3\rho \ln t + \rho (\rho +1) \geq 2 r^2  (\ln t)^{2(1-\rho)},
$$
which is satisfied for any $r\leq 1$ and any $t \geq e$.
\begin{itemize}
\item  If $\rho =1$, then 
$
p (t) = \exp \left(\displaystyle\int_{t_0}^{t}\frac{1}{s (\ln s)^{\rho}}ds \right)= \exp \left(\displaystyle\int_{\ln t_0 }^{\ln t}\frac{du}{x} \right)=\frac{\ln t}{\ln t_0 },
$
\\ and for $r=\frac{1}{3}$, we get  
$f (x(t))-\min_\mathcal{H}f = \mathcal O \left(\frac{1}{(\ln t )^{\frac{2}{3}}} \right).$
\item  If $0 \leq \rho < 1$, then 
$
p(t)= \exp \left(\displaystyle\int_{\ln t_0 }^{\ln t}\frac{1}{u ^{\rho}}du \right) = \exp \left( \frac{1}{1-\rho}
\left(  (\ln t )^{1- \rho} -  (\ln t_0 )^{1- \rho} \right)\right),
$ 
and, for $r=\frac{1}{3}$, we also get 
$f (x(t))-\min_\mathcal{H}f= \mathcal O \left(\frac{1}{\exp \left( \frac{2}{3(1-\rho)}
  (\ln t )^{1- \rho} \right)} \right).$
  \end{itemize}
  \end{remark}

 \subsection{The case  $b(t)=c_0t^{q}$ and $\gamma(t)=\frac{\alpha}{t}$.}
  When $b(t)=c_0t^{q}$ and $\gamma(t)=\frac{\alpha}{t}$  where $\alpha>0$ and $q\geq 0$, we first observe that 
  $p (t)=\exp\left(\int_{t_{0}}^{t}\gamma(s)ds\right)=\left(\frac{t}{t_0}\right)^\alpha .
  $
  The second-order continuous system  becomes:
    \begin{equation}\label{alpha}
    \ddot{x}(t)+\frac{\alpha}{t}\dot{x}(t)+c_0t^{q}\nabla f(x(t))=0.
\end{equation} 
Applying Theorem \ref{main}, we obtain  the following new result.
   \begin{theorem}\label{cor1.6}
   Let $x(\cdot)$  be a solution trajectory of {\rm (\ref{alpha})} with $\alpha>1$ and $q\geq 0$. 
  Suppose that $1<\alpha\leq3+q$. Then,
\begin{equation}\label{eq18}
f(x(t))-\min_\mathcal{H}f = \mathcal O\left(\frac{1}{t^{\frac{2\alpha+q}{3}}}\right),\hbox{ as } \, t \to +\infty.
\end{equation}
\end{theorem}
\begin{remark}
Taking $q=0$, a direct application of the above result 
covers the results obtained in  \cite{redon,boyd} (case $\alpha \geq 3$), and in \cite{AAD,att1}, (case $\alpha \leq 3$). It suffices to take $\gamma (t)=\frac{\alpha}{t}$ and $r=\frac1{\alpha}$.
More precisely, we get  :
\begin{itemize}
\item if  $0<\alpha\leq3$ then
$f(x(t))-\min_\mathcal{H}f = \mathcal O(t^{\frac{-2\alpha}{3}})$, 

\smallskip

\item if $\alpha>3$ then $f(x(t))-\min_\mathcal{H}f = \mathcal O(\frac{1}{t^{2}})$.
\end{itemize}
\end{remark}
 \subsection{The case $b(t)=e^{\mu t^{q}}$ and $\gamma(t)=\frac{\alpha}{t^{1-q}}$.}
     Suppose that $\mu \geq 0\;, 0\leq q\leq1$  and $\alpha> 0 $. This will allow us to obtain the following exponential convergence rate of the values.
\begin{theorem}\label{thm6.4}
Let $x: [t_{0},+\infty[\longrightarrow \mathcal{H}$  be a solution trajectory of 
 \begin{equation}\label{alpha1}
  \ddot{x}(t)+\frac{\alpha}{t^{1-q}}\dot{x}(t)+e^{\mu t^{q}}\nabla f(x(t))=0.
  \end{equation}  
Suppose that $\alpha\leq \mu q$, then, as $t \to +\infty$
  $$f(x(t))-\min_\mathcal{H}f=\mathcal O\left(e^{-\frac{2\alpha+\mu q}{3}t^{q}}\right).$$
\end{theorem}

\begin{remark} a)  For $q=\mu=0$, (\ref{alpha1}) reduces to the system initiated in  \cite{boyd}, \ie 
$$
  \ddot{x}(t)+\frac{\alpha}{t}\dot{x}(t)+\nabla f(x(t))=0.
  $$
Just assuming $\alpha>0$, we obtain  
    $\displaystyle\lim_{t\rightarrow +\infty}\left(f(x(t))-\min_\mathcal{H}f\right)=0$.
 
\noindent b) For $q=\frac{1}{2}$ we get
  \begin{itemize}
    \item If $\alpha\leq \mu$, then
    $f(x(t))-\min_\mathcal{H}f=\mathcal O\left(e^{-\frac{2(2\alpha+\mu )}{3})\sqrt{t}}\right).$
    \item If $\alpha\geq \mu$, then
    $f(x(t))-\min_\mathcal{H}f=\mathcal O\left(e^{-2\mu \sqrt{t}}\right).$
  \end{itemize}
\end{remark}
c) For $q=1$, direct application of  Theorem \ref{thm6.4} gives:
\begin{corollary}[\textbf{Linear convergence}]\label{convlin}
Let $x: [t_{0},+\infty[\rightarrow \mathcal{H}$  be a solution trajectory of  
\begin{equation}\label{alpha1c}
  \ddot{x}(t)+\alpha\dot{x}(t)+e^{\mu t}\nabla f(x(t))=0.
  \end{equation} 
If $\alpha\leq \mu ,$ then \;
  $f(x(t))-\min_\mathcal{H}f=\mathcal O\left(e^{-\frac{2\alpha +\mu }{3}t}\right).$
\end{corollary}
 Let us illustrate these results. Take $f(x_1,x_2):=\frac12\left(x_1^2+x_2^2\right)-\ln(x_1x_2)$, which is a strongly convex function. Trajectories of 
$$
\ddot{x}(t)+\alpha\dot{x}(t)+e^{\mu t}\nabla f(x(t))+ ce^{\nu t}\nabla^2 f(x(t))\dot{x}(t)=0,
$$
 corresponding to different values of the parameters $\alpha$, $\mu$, $\nu$, and $c$,  are plotted in Figure \ref{strong-convex}
 \footnote{From  Scilab version 6.1.0  
 http://www.scilab.org as an open source software}.
The parameter $c$ shows the importance of the Hessian-damping.
\begin{figure}[h]
\begin{center}
\includegraphics[width=12.5cm]{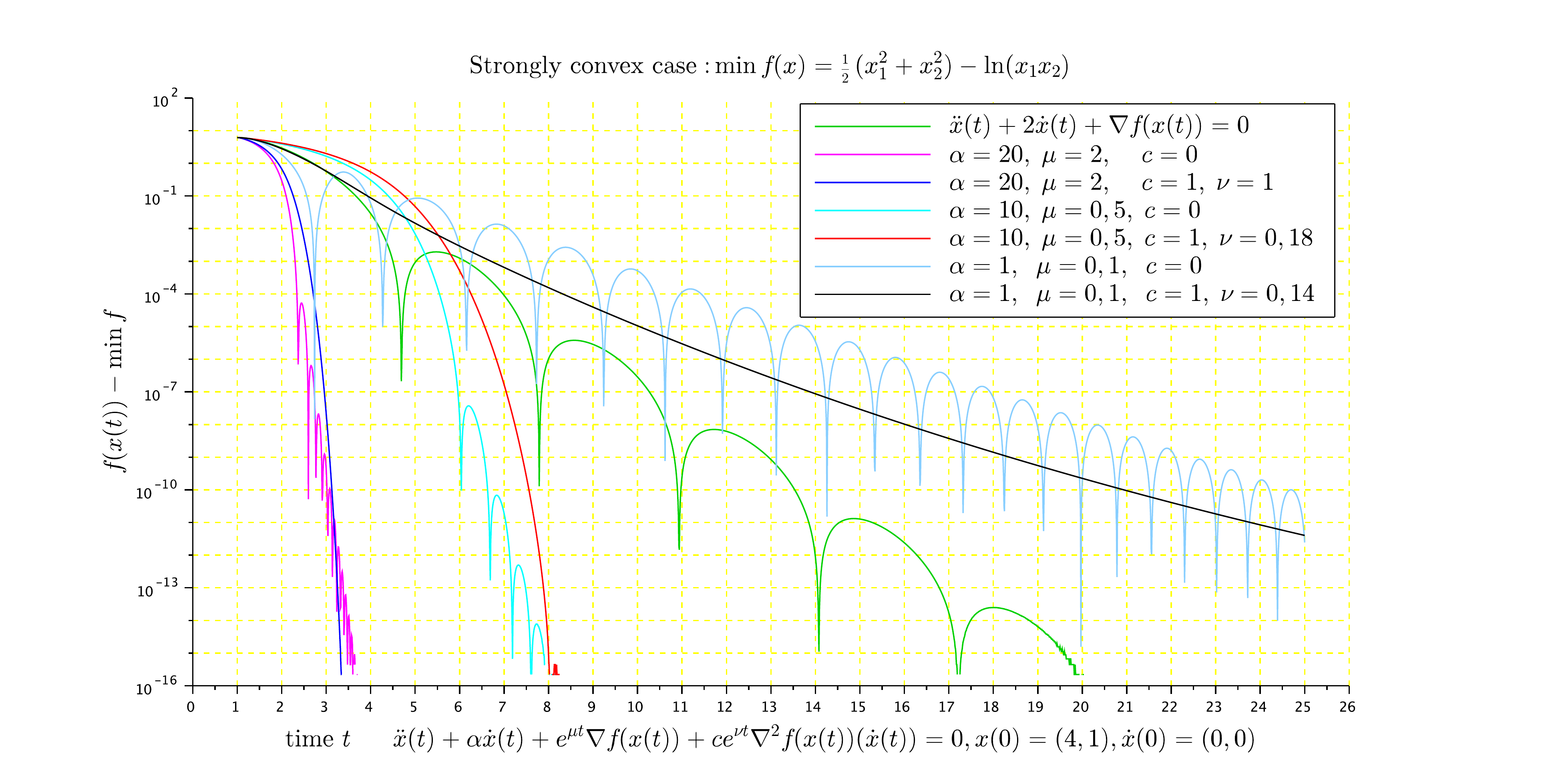}
\caption{Evolution of $f (x(t)) - \min f $ for  solutions of   (\ref{alpha1c}), (\ref{strong}), and $f(x_1,x_2)=\frac12\left(x_1^2+x_2^2\right)-\ln(x_1x_2)$.}
\label{strong-convex}
\end{center}
\end{figure}
\subsection{Numerical comparison}
Figure \ref{table1-compatison} summarizes our convergence results, according to the behavior of the parameters  $\gamma(t)$, $\beta(t)$, $b(t)$.
Let's comment on them and compare them, separately considering $f$ to be strongly convex or not.
\begin{figure}
\begin{center}
\begin{tabular}{|c|c|c||c||c|}
     \hline
$ \gamma(t) $ & $\beta(t)$ & $b (t)$ & $f(x(t))-\min f $& Reference \\     \hline \hline
Cte	& 0 & 1	& $\mathcal O\left( t^{-1}\right)	$	& (1964) \cite{polyak}\\ \hline
Cte	& Cte & 1	& $\mathcal O\left( t^{-1}\right)	$	& (2002) \cite{AABR}\\ \hline
$\alpha/t$ & 0 & 1 &  \begin{tabular}{l} $\mathcal O\left(t^{-\frac23\alpha}\right)$  if  $0<\alpha\leq 3$\\      
$\mathcal O\left(t^{-2}\right)\;\;\;\;$  if  $\alpha\geq 3$  \end{tabular} & \begin{tabular}{l} (2019) \cite{att1}\\  (2014)  \cite{boyd} \end{tabular}  \\ \hline
$\alpha/t$ & Cte &  1	& $\mathcal O\left( t^{-2}\right)$	 if  $\alpha\geq 3,\beta>0$ & (2016)  \cite{APR}\\ \hline
$\gamma(t)$ & 0 & $b(t)$	&  \begin{tabular}{l} $\mathcal O\left(\left(p(t)\int_t^{+\infty}(p(s))^{-1}ds\right)^{-2} (b(t))^{-1}\right)$ \\     
where  $p(t) \eqdef \exp\left(\int_{t_0}^t\gamma(s)ds\right)$ \end{tabular} & (2019) \cite{att1}\\ \hline
$\alpha/t$ & $\beta(t)$ &  $b(t)$	& $\mathcal O\left(\left(t^2 b(t)-\dot{\beta}(t) -\dfrac{\beta(t)}{t}\right)^{-1}\right)$
& (2020)  \cite{ACFR}\\ \hline
  \end{tabular}
 \end{center}
    \caption{Convergence rate  of $f(x(t))-\min f$ for instances of  Theorem \ref{th1aa} and general  $f$.}\label{table1-compatison}
\end{figure}
\subsubsection{Strongly convex  case} 
Suppose that $f$ is $s$-strongly convex. Following Polyak's \cite{polyak}, the system 
 \begin{equation}\label{strong}
    \ddot{x}(t)+2\sqrt{s}\dot{x}(t)+\nabla f({x}(t))=0
\end{equation}
provides the  linear  convergence rate    $ f (x(t)) - \inf_{\mathcal{H}} f   \leq    Ce^{-\sqrt{s}t}$,
see also \cite[Theorem~2.2]{Siegel}. In the presence of an additional Hessian-driven damping term 
\begin{equation}\label{dyn-sc}
\ddot{x}(t) + 2\sqrt{s} \dot{x}(t) + \beta \nabla^2 f (x(t))\dot{x}(t) + \nabla f (x(t)) = 0 \;\;\;(\beta\geq 0)
\end{equation}
  a related  linear rate of convergence  can be found in \cite[Theorem~7]{ACFR}. 
Let us insist on the fact that, in Corollary \ref{convlin}, we obtain a linear convergence rate for a general convex differentiable function $f$.  
In Figure \ref{strong-convex}, for  the strongly convex function 
$
f(x_1,x_2)=\frac12\left(x_1^2+x_2^2\right)-\ln(x_1x_2),
$
 we can observe that some values of  $ \mu $ give a better speed of convergence of $f (x (t)) - \min f $.
We can also note that for $ \mu $ correctly set,  the system (\ref{alpha1c}) provides a better linear convergence rate  than  the system (\ref{strong}).  
\subsubsection{Non-strongly convex  case} 
We illustrate our results on the following simple example of a non strongly convex  minimization problem, with non unique solutions.
\begin{equation}\label{examp01}
\min_{\R^2}f (x_1,x_2)=\frac12(x_1+10^3x_2)^2.
 \end{equation}
 From Figure \ref{non-strong-convex} we get the following properties:\\ 
a) The  convergence rate of the values is in accordance with Figure \ref{table1-compatison}.\\
b) The system (\ref{alpha1c}) is  best  for its linear convergence of values.\\
c) The  Hessian-driven damping reduces the oscillations of the  trajectories.
\begin{figure}[h]
\begin{center}
\includegraphics[width=12cm]{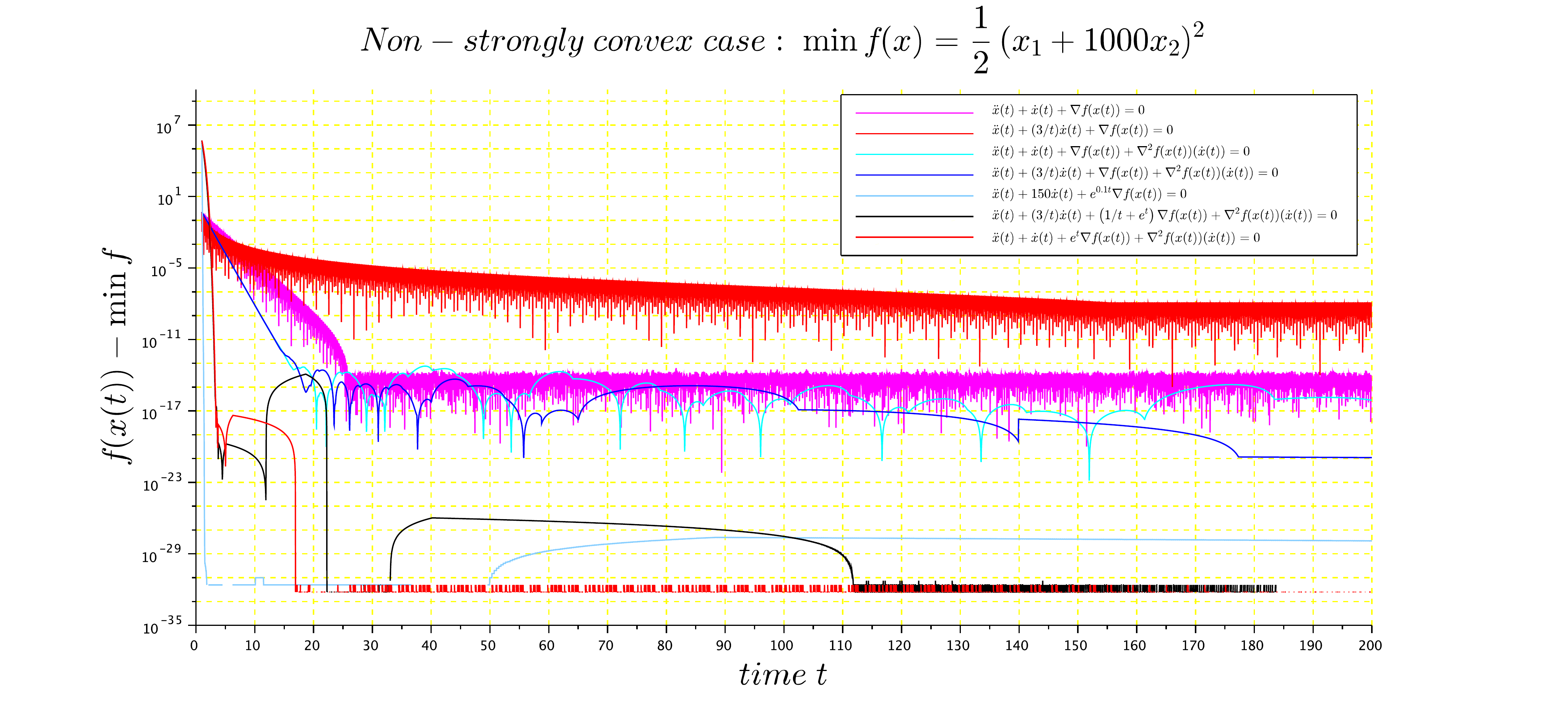}
\caption{ Evolution   of  $f (x(t)) - \min f $ for  systems in Figure \ref{table1-compatison}, and  $f(x_1,x_2)=\frac12\left(x_1^2+10^3x_2^2\right)$.  
}
\label{non-strong-convex}
\end{center}
\end{figure}
\section{Conclusion, perspectives}
Our study is one of the first works to simultaneously consider the combination of three basic techniques for the design of fast converging inertial dynamics in convex optimization: general viscous damping (and especially asymptotic vanishing damping in relation to the Nesterov accelerated gradient method), 
Hessian-driven damping which has a spectacular effect on the reduction of the oscillatory aspects (especially for ill-conditionned minimization problems), and temporal rescaling. We have introduced a system of equations-inequations whose solutions provide  the coefficients of a general Lyapunov functions for these dynamics. We have been able to encompass most of the existing results and find new solutions for this system, thus providing new Lyapunov functions. Also, we have been able to explain the  mysterious coefficients which have been used in recent algorithmic developements, and which were just justified until now by the simplification of complicated calculations. Finally, by playing on fast rescaling methods, we have obtained linear convergence results for general convex functions.
This work provides a basis for the development of corresponding algorithmic results.

\end{document}